\documentclass[twoside, 11pt]{article}
\usepackage{amsmath,amssymb,amsthm,mathrsfs}
\usepackage[margin=2.5cm]{geometry} % showframe,
\usepackage{lipsum}
\usepackage{titlesec,hyperref}
\usepackage{fancyhdr}
\usepackage[numbers,sort&compress]{natbib}
\usepackage{color}
\usepackage{stmaryrd}

\fancypagestyle{plain}
{
\fancyhf{}

}

\titleformat{\subsection}{\it}{\thesubsection.\enspace}{1.5pt}{}
\titleformat{\subsubsection}{\it}{\thesubsubsection.\enspace}{1.5pt}{}

\newtheorem{theo}{Theorem}[section]
\newtheorem{lemm}[theo]{Lemma}

\newtheorem{coro}[theo]{Corollary}
\newtheorem{prop}{Proposition}[section]
\newtheorem{rema}{Remark}[section]
\numberwithin{equation}{section}

\allowdisplaybreaks

\def\ma{\mathcal{A}}
\def\l{\left\|}
\def\r{\right\|}

\begin{document}
\title{Decay rates for the Viscous Incompressible MHD with and without Surface Tension}
\author{ Boling Guo$^a$,  Lan Zeng$^{b,*}$, Guoxi Ni$^a$ }
\date{}
\maketitle
\begin{center}
\begin{minipage}{120mm}
\emph{\small $^a$Institute of Applied Physics and Computational Mathematics, China Academy of Engineering Physics,
 Beijing, 100088, P. R. China \\
$^b$Graduate School of China Academy of Engineering Physics, Beijing, 100088, P. R. China \\}
\end{minipage}
\end{center}
{\bf Abstract.}
In this paper, we consider a layer of a viscous incompressible electrically  conducting fluid interacting with the magnetic filed in a horizontally periodic setting. The upper boundary bounded by a free boundary and below bounded by a flat rigid interface.
We prove the global well-posedness of the problem for both the case with and without surface tension. Moreover, we show that the global solution decays to the equilibrium exponentially in the case with surface tension, however the global solution decays to the equilibrium at an almost exponential rate in the case without surface tension.

\bigskip
\noindent {\bf Key Words:}
 MHD, global well-posedness, decay rates, with and without surface tension
 
\bigskip
\noindent {\bf 2010 Mathematics Subject Classification:}~35Q35,~35R35,~76N10,~76W05

 \footnotetext{Corresponding Author: zenglan1206@126.com (L.Zeng).

~~This work is supported by NSFC under grant numbers 11731014, 11571254.}

\section{Introduction}
\subsection{Formulation in Eulerian Coordinates}
We consider the motion of an viscous incompressible electrically conducting fluid interacting with the magnetic field in a 3D moving domain
\begin{equation}
\Omega(t)= \{ y\in \Sigma\times R| -1<y_{3} <\eta(y_{1},y_{2},t) \}.
\end{equation}
 We assume $\Omega(t)$ is horizontally periodic by setting $\Sigma= (L_{1}\mathbb{T}) \times (L_{2}\mathbb{T})$ for $\mathbb{T}=\mathbb{R}/\mathbb{Z}$ the 1-torus and $L_{1},L_{2}>0$ periodicity lengths. The upper boundary $\{y_{3}= \eta(y_{1},y_{2},t)\}$ is a free surface that is the graph of the unknown function $\eta: \Sigma\times \mathbb{R}^{+}\rightarrow \mathbb{R}$. The dynamics of the fluid is described by the velocity, the pressure and the magnetic field, which are given for each $t\geq0$ by $\tilde{u}(t,\cdot):\Omega(t)\rightarrow \mathbb{R}^3$, $\tilde{p}(t,\cdot):\Omega(t)\rightarrow \mathbb{R}$ and $\tilde{B}(t,\cdot):\Omega(t)\rightarrow \mathbb{R}^3$, respectively. For each $t>0$, $(\tilde{u},\tilde{p},\tilde{B},\eta)$ is required to satisfy the following free boundary problem for the incompressible viscid and resistive magnetohydrodynamic equations (MHD):
 \begin{equation}\label{1.2}
 \left\{
 \begin{aligned}
&\partial_{t}\tilde{u} + \tilde{u} \cdot \nabla \tilde{u}- \mu \Delta \tilde{u}+\nabla \tilde{p}= \tilde{B}\cdot\nabla \tilde{B} , ~~~~~~~~~{\rm in}~\Omega(t)\\
&{\rm div} \tilde{u}=0,~~~~~~~~~~~~~~~~~~~~~~~~~~~~~~~~~~~~~~~~~~~~{\rm in}~\Omega(t)\\
&\partial_{t} \tilde{B}+ \tilde{u}\cdot \nabla \tilde{B}- \kappa \Delta \tilde{B}=\tilde B\cdot\nabla u,~~~~~~~~~~~~~~~{\rm in}~\Omega(t)~~\\
&{\rm div}B=0~~~~~~~~~~~~~~~~~~~~~~~~~~~~~~~~~~~~~~~~~~~~~{\rm in}~\Omega(t)\\
&\partial_t\eta=u_3-u_1\partial_{y_1}\eta-u_2\partial_{y_2}\eta ~~~~~~~~~~~~~~~~~~~~~{\rm on} \{y_3=\eta(t,y_1,y_2)\}\\
&(\tilde p I-\mu \mathbb{D}(\tilde u))\nu=g\eta\nu+\sigma M\nu,~\tilde B=\bar B~~~~~~~~{\rm on} \{y_3=\eta(t,y_1,y_2)\}\\
&\tilde{u}=0,~\tilde B=\bar B~~~~~~~~~~~~~~~~~~~~~~~~~~~~~~~~~~~~~~~{\rm on} \{y_3=-1\}.\\
\end{aligned}
\right.
\end{equation}
Here $\nu$ is the outward-pointing unit normal on $\{y_3=\eta\}$, $\bar B$ is the constant magnetic field in the outside of the fluid. $\mu>0,~\kappa >0$  are the kinematic viscosity and magnetic diffusion coefficient, respectively. The first four equations in \eqref{1.2} are the usual viscous incompressible MHD equations. The fifth equation implies that the free surface is advected with the fluid. The sixth equation is the balance of the stress on the free surface, where $I$ is the $3\times3$ identity matrix, and $(\mathbb{D}\tilde u)_{ij}=\partial_i\tilde u_j+\partial_j\tilde u_i$ is the symmetric gradient of $\tilde u$. The tensor $(\tilde p I-\mu\mathbb{D}(\tilde u))$ is known as the viscous stress tensor, $g$ is the strength of gravity. $M$ is the mean curvature of the free surface and is given by $M=\partial_i(\partial_i\eta/\sqrt{1+|D\eta|^2})$. Note that, in \eqref{1.2}, we have shifted the gravitational forcing to the free boundary and eliminated the constant atmospheric pressure, $P_{atm}$, the magnetic pressure $|\tilde B|^2/2$ and the constant outside magnetic pressure $|\bar B|^2/2$, in the usual way by adjusting the actual pressure $\bar p$ according to
\begin{equation}
\tilde p=\bar p+gy_3-P_{atm}+|\tilde B|^2/2-|\bar B|^2/2.
\end{equation}
To complete the statement of the problem, we assume the problem satisfies the following initial conditions.
\begin{equation}
\eta(0)=\eta_0,~ \tilde u(0)=u_0,~\tilde B(0)=B_0,
\end{equation}
furthermore, we will assume $\eta_0>-1$, which means at the initial time the boundary do not intersect with each other.

In the global well-posedness theory of the problem \eqref{1.2}, we suppose that the initial surface function satisfies the following ``zero average'' condition
\begin{equation} \label{1.3}
\frac{1}{L_{1} L_{2}} \int_{\Sigma} \eta_{0}=0.
\end{equation}
Notice that for sufficiently regular solutions to the periodic problem, the condition \eqref{1.3} persists in time, indeed, according to $\partial_{t} \eta= \tilde{u}\cdot \nu \sqrt{1+ (\partial_{y_1}\eta)^2+(\partial_{y_2}\eta)^2}$,
\begin{equation}
\frac{d}{dt} \int_{\Sigma} \eta = \int_{\Sigma} \partial_{t} \eta = \int_{\{y_3=\eta(t,y_1,y_2)\}} \tilde{u}\cdot \nu= \int_{\Omega(t)} {\rm div} \tilde{u}=0,
\end{equation}
which allows us to apply Poincar${\rm\acute{e}}$'s inequalities on $\Sigma$ for $\eta$ for all $t\geq0$.

\subsection{Formulation in flattening coordinates}

The Moving free boundary and the subsequent change of the domain generate plentiful mathematical difficulties. To overcome these, as usual, we will use a coordinate transformation to flatten the free surface. Here we will not use a Lagrangian coordinate transformation, but rather a flatting transformation introduced by Beale \cite{Beale}. To this end, we consider the fixed equilibrium domain
\begin{equation}
\Omega:=\{x\in\Sigma\times\mathbb{R}|-1<x_3<0\},
\end{equation}
for which we will write the coordinates as $x\in \Omega$. We will think of $\Sigma$ as the upper boundary of $\Omega$, and we will write $\Sigma_{-1}:= \{x_{3}=-1 \}$ for the lower boundary. We continue to view $\eta$ as a function on $\Sigma \times R^{+}$. We then define
\begin{equation*}
\bar{\eta}:= \mathcal{P}\mathcal{\eta} ={\rm~harmonic~extension~of~\eta~into~the~lower~half~space,}
\end{equation*}
where $\mathcal{P}\eta$ is defined by \eqref{7.1}. The harmonic extension $\bar{\eta}$ allows us to flatten the coordinate domain via the mapping
\begin{equation}\label{1.4}
\Omega \ni x \mapsto (x_{1},x_{2}, x_{3}+ \bar{\eta}(x,t)(1+x_{3})):= \Phi(x,t)= (y_{1},y_{2},y_{3})\in \Omega(t),
\end{equation}
Note that $\Phi(\Sigma,t)= \{y_{3}=\eta(y_{1},y_{2},t)\} $ and $\Phi(\cdot,t)|_{\Sigma_{-1}}= Id_{\Sigma_{-1}}$, i.e. $\Phi$ maps $\Sigma$ to the free surface and keeps the lower surface fixed. We have
\begin{equation}
\nabla \Phi= \left(
\begin{aligned}
&1~~~0~~~0 \\
&0~~~1~~~0 \\
&A~~B~~J
\end{aligned}
\right)
~~{\rm and}~~
\mathcal{A} := (\nabla \Phi^{-1})^{T}=\left(
\begin{aligned}
&1~~0~~-AK \\
&0~~1~~-BK \\
&0~~0~~~~~K
\end{aligned}
\right)
\end{equation}
for
\begin{equation}\label{1.10}
A= \partial_{1} \bar{\eta} \widetilde{b},~~B= \partial_{2}\bar{\eta} \widetilde{b}, ~~\widetilde{b}=(1+x_{3})
\end{equation}
\begin{equation}\label{1.11}
J=1+ \bar{\eta}+ \partial_{3} \bar{\eta}\widetilde{b},~~K=J^{-1}.
\end{equation}
Here $J={\rm det}( \nabla \Phi)$ is the Jacobian of the coordinate transformation. If $\eta$ is sufficiently small in an appropriate Sobolev space, then the mapping is a diffeomorphism. It allows us to transform the problem to one on the fixed spatial domain. Note that the following useful relation will be frequently used throughout this paper:
\begin{equation}\label{1.13}
\partial_k(J\mathcal{A}_{jk})=0.
\end{equation}
Without loss of generality, we will assume that $\mu=g=\kappa=1$. Indeed, a standard scaling argument allows us to scale so that $\mu=g=\kappa=1$. Furthermore, we define the transformed quantities as
\begin{equation*}
u(t,x) := \tilde{u}(t,\Phi(t,x)),~~p(t,x) :=\tilde{p}(t,\Phi(t,x)), ~~b(t,x) :=\tilde{B}(t,\Phi(t,x))-\bar B.
\end{equation*}
In the new coordinates, \eqref{1.2} can be written as
\begin{equation}\label{1.14}
\left\{
\begin{aligned}
&\partial_{t} u-\partial_{t}\bar{\eta}\tilde{b}K \partial_{3} u + u\cdot \nabla_{\ma } u-\Delta_{\ma}u+\nabla_\ma p=(b+\bar B)\cdot\nabla_{\ma}b,~~{\rm in}~\Omega~ \\
&{\rm div}_{\ma} u=0,~~~~~~~~~~~~~~~~~~~~~~~~~~~~~~~~~~~~~~~~~~~~~~~~~~~~~~~~~~~~~~~~~{\rm in}~\Omega~ \\
&\partial_{t}b-\partial_{t}\bar{\eta}\tilde{b}K \partial_{3}b + u\cdot \nabla_{\ma}b -\Delta_{\ma} b=(b+\bar B)\cdot\nabla_{\ma}u,~~~~~~~~~~~~~~{\rm in}~\Omega~ \\
&{\rm div}_{\ma} b=0,~~~~~~~~~~~~~~~~~~~~~~~~~~~~~~~~~~~~~~~~~~~~~~~~~~~~~~~~~~~~~~~~~~{\rm in}~\Omega~ \\
&(pI-\mathbb{D}_{\ma} u)\mathcal{N}=\eta \mathcal{N} +\sigma M\mathcal{N} ,~~b=0~~~~~~~~~~~~~~~~~~~~~~~~~~~~~~~~{\rm on}~\Sigma,~\\
&\partial_{t} \eta+ u_{1} \partial_{1}\eta +u_{2} \partial_{2}\eta =u_{3}, ~~~~~~~~~~~~~~~~~~~~~~~~~~~~~~~~~~~~~~~~~~~~~{\rm on}~\Sigma~\\
&u=0,~b=0,~~~~~~~~~~~~~~~~~~~~~~~~~~~~~~~~~~~~~~~~~~~~~~~~~~~~~~~~~~~~~~{\rm on}~\Sigma_{-1}\\
&u(x,0)=u_0(x),~b(x,0)=b_0(x),~\eta(x_1,x_2,0)=\eta_0(x_1,x_2).
\end{aligned}
\right.
\end{equation}
Here we have written the differential operators $\nabla_{\ma}$, ${\rm div}_{\ma}$, and $\Delta_{\mathcal{A}}$ with their actions given by $(\nabla_{\ma} f)_{i} := \ma_{ij}\partial_{j} f$, ${\rm div}_{\ma} X=\ma_{ij}\partial_{j} X_{i}$, and $\Delta_{\mathcal{A}} f= {\rm div}_{\mathcal{A}} \nabla_{\mathcal{A}} f$ for approximate $f$ and $X$; for $u\cdot \nabla_{\mathcal{A}}u$ we mean $(u\cdot \nabla_{\mathcal{A}}u)_{i} := u_{j}\mathcal{A}_{jk} \partial_{k} u_{i}$. We have also written $(\mathbb{D}_{\mathcal{A}} u)_{ij}= \ma_{ik} \partial_{k} u_{j}+ \ma_{jk} \partial_{k} u_{i}$. Also, $\mathcal{N}:=-\partial_1\eta e_1-\partial_2\eta e_2+e_3$ denotes the non-unit normal on $\Sigma$.

\subsection{Related works}

The problem of free boundary in fluid mechanics has been deeply studied in the field of mathematics, and there are a huge number of impressive results. Here, we only introduce briefly some works related to our problem.

When $B=0$ in model \eqref{1.2}, it reduces to the well known viscous surface wave problem. The reduced problem without surface tension was studied firstly by Beale \cite{Beale}, in which the local well-posedness in the Sobolev spaces had been proved. And Sylvester studied the global well-posedness by using Beale's method in \cite{Sylve}. For the periodic case, Hataya \cite{Hataya} proved the global existence of small solutions with an algebraic decay rate. In \cite{Guo,Guoo,Guooo}, Guo and Tice used a new two-tier energy method to proved the local well-posedness, the global solution decay to the equilibrium at an algebraic decay rate in the non-periodic case and decay to equilibrium at an almost exponential rate in the periodic case, respectively. For the case with surface tension, the global well-posedness was proved in the Sobolev spaces by Beale \cite{Bealee}, and Bae \cite{Bae} the globle solvability in Sobolev spaces via the energy method. Beale et.al \cite{Bealeee} and Nishida et.al \cite{Nishida} proved that the global solution obtained in \cite{Bealee} decays at an optimal algebraic rate in the non-periodic case and decays at an exponential decay rate in the periodic case, respectively. Tani \cite{Tani} and Tani et.al \cite{Tanii} considered the solvability of the problem with or without surface tension under the Beale-Solonnikov's function framework. Furthermore, in \cite{Tan} Tan and Wang proved the zero surface tension limit within a local time interval and the global one under the small initial data. Furthermore, in \cite{Kim,Tice,Ticee} Tice et.al. researched the effect of the more general surface tension on the decay rate for the viscous surface waves problem.

Correspondingly, for the case $B\neq0$, namely, the free boundary problem for the viscous MHD equations, there are only a few results. The local-well posedness for the viscous MHD equations in a bounded variable domain with surface tension was proved by Padula et.al. in \cite{Padula}, and the small initial data global solvability for the same model was obtained by Solonnikov et.al. in \cite{Solonnikov}. In \cite{Lee}, Lee used the method developed by Masmoudi \cite {Masmoudi} to derive the vanishing viscosity limit with surface tension under the initial magnetic field is zero on the free boundary and in vacuum. Recently, for the model \eqref{1.2}, Wang and Xin \cite{Wanggg} studied the 2D case with $\mu=0$ and $\sigma>0$, and they proved the global solution decays to the equilibrium at an almost exponentially decay rate, in which they use the structure of the equations sufficiently to find a damping structure for the fluid vorticity which plays an important role to close the energies estimates.

Motivated by these articles mentioned above, in this paper, we focus on the free boundary problem for the incompressible viscous and resistive MHD equations both the case with and without surface tension, in which we mainly discuss the effect of surface tension on the decay rate of system \eqref{1.2}.

In this paper, for the case without surface tension, we mainly use the method mentioned in \cite{Tice} to overcome the lack of regularity for $\eta$. However, we have not use the structure ${\rm div}_\ma u=0$ to write $\partial_3u_3=-(\partial_2u_1+\partial_2u_2)+G^2$  to improve the full dissipation estimates of $u$, where $G^2$ are some quadratic nonlinearities. Here, we use a much more simple method used in \cite{Tan} to obtain the full dissipation estimates for $u$ and $p$, in which they had a crucial observation that they can get higher regularity estimates of $u$ on the boundary $\Sigma$ only from the horizontal dissipation estimates.

\subsection{Some definitions and notations}
Now, we state some definitions and notations that will be used throughout this paper. The Einstein convention of
summing over repeated indices for vector and tensor operations. In this paper, $C>0$ will denote a generic constant that can depend on $N$ and $\Omega$, but does not depend on the initial data and time. We refer to such constants as ``universal'', which are allowed to change from line to line. We use the notation $A\lesssim B$ to mean that $A\leq CB$ where $C>0$ is a universal constant. We will use $\mathbb{N}^{1+m}=\{\alpha=(\alpha_0,\alpha_1,\cdots,\alpha_m)\}$ to emphasize that the 0-index term is related to temporal derivatives. For $\alpha\in\mathbb{N}^{1+m}$ we write $\partial^\alpha=\partial^{\alpha_0}_t\partial^{\alpha_1}_1\cdots\partial^{\alpha_m}_m$. For just spatial derivatives we write $\mathbb{N}^{m}$, namely $\alpha_0=0$. We define the parabolic counting of such multi-indices by writing $|\alpha|=2\alpha_0+\alpha_1+\cdots+\alpha_m$. We will write $Df$ for the horizontal gradient of $f$, that is, $Df=\partial_1fe_1+\partial_2fe_2$, while $\nabla f$ will denote the usual full gradient.

We write $H^k(\Omega)$ with $k\geq0$ and $H^s(\Sigma)$ with $s\in\mathbb{R}$ for the usual Sobolev spaces, and we will denote $H^0=L^2$. In this paper, for simplicity, we will avoid writing $H^k(\Omega)$ or $H^s(\Sigma)$ and write only $\|\cdot\|_k$. When we write $\|\partial^j_tu\|_k$, it means that the space is $H^k(\Omega)$ and when we write $\|\partial^j_t\eta\|_k$, it will means that the space is $H^k(\Sigma)$.

For a given norm $\|\cdot\|$ and integers $k,m\geq0$, we introduce the following notation for sums of spatial derivatives:
\begin{equation}
\| D^{k}_m f\|^{2} := \sum_{\alpha\in \mathbb{N}^{2}, m\leq|\alpha|\leq k} \|\partial^{\alpha} f\|^{2} ~~~
{\rm and}~~~\| \nabla^{k}_m f\|^{2} := \sum_{\alpha\in \mathbb{N}^{3}, m\leq|\alpha|\leq k} \|\partial^{\alpha} f\|^{2}
\end{equation}
The convention we adopt in this notation is that $D$ refers to only horizontal spatial derivatives, while $\nabla$ refers to full spatial derivatives. For space-time derivatives we add bars to our notation:
\begin{equation}
\| \bar{D}^{k}_m f\|^{2} := \sum_{\alpha\in \mathbb{N}^{1+2}, m\leq|\alpha|\leq k} \|\partial^{\alpha} f\|^{2} ~~~
{\rm and}~~~\| \bar{\nabla}^{k}_m f\|^{2} := \sum_{\alpha\in \mathbb{N}^{1+3}, m\leq |\alpha|\leq k} \|\partial^{\alpha} f\|^{2}
\end{equation}
When $k=m\geq0$, we denote
 \begin{equation}
 \begin{aligned}
\| D^{k} f\|^{2}=\| D^{k}_k f\|^{2},~~\| \nabla^{k} f\|^{2}=\| \nabla^{k}_k f\|^{2},\\
\| \bar{D}^{k} f\|^{2}=\| \bar{D}^{k}_k f\|^{2},~~\| \bar{\nabla}^{k} f\|^{2}=\| \bar{\nabla}^{k}_k f\|^{2}.
\end{aligned}
\end{equation}
The rest of this paper unfolds as follows. In section 2, we first define the energies and dissipations, and then state our main results. In section 3 we prove some preliminary lemmas that we will use in our a priori estimates. In section 4, we complete the a priori estimates for the case $\sigma>0$. In section 5, we closed the a priori estimates for the case $\sigma=0$.

\section{Main Results}

We first state the result for \eqref{1.14} in the case $\sigma>0$. Firstly, we define some energy functions in this case. We define the energy as
\begin{equation}\label{12.1}
\begin{aligned}
\mathcal{E} :=&\|u\|^{2}_{2} + \|\partial_{t} u\|^{2}_{0}+ \|b\|^{2}_{2}+ \|\partial_{t}b\|_{0}
+ \|p\|^{2}_{1}+ \|\eta\|^{2}_{3}~~\\
&+ \|\partial_{t} \eta\|^{2}_{{3}/{2}}+ \|\partial_{t}^{2} \eta\|^{2}_{-{1}/{2}},
\end{aligned}
\end{equation}
and define the dissipation as
\begin{equation}\label{12.2}
\begin{aligned}
\mathcal{D} :=&\|u\|^{2}_{3} + \|\partial_{t} u\|^{2}_{1}
 + \|b\|^{2}_{3}+ \|\partial_{t} b \|_{1}+ \|p\|^{2}_{2}+ \|\eta\|^{2}_{{7}/{2}}~~\\
&+ \|\partial_{t} \eta\|^{2}_{{5}/{2}}+ \|\partial_{t}^{2}\eta\|^{2}_{1/2}.
\end{aligned}
\end{equation}
In the case $\sigma>0$, the global well-posedness result is stated as follows.
\begin{theo}\label{thm:0102}
For $\sigma>0$, we assume that the initial datum $u_{0}\in H^{2}(\Omega)$, $\eta_{0}\in H^{3}(\Sigma)$, $b_{0}\in H^{2}(\Omega)$ and satisfy some appropriate compatibility conditions as well as the zero-average condition \eqref{1.3}. Then there exists a universal constant $\kappa>0$ such that, if
\begin{equation*}
\|u_{0}\|^{2}_{2}+ \|\eta_{0}\|^{2}_{3}+ \|b_{0}\|^2_{2}\leq \kappa,
\end{equation*}
then, for all $t\geq0$, there exists a unique strong solution $(u,p,\eta,b)$ to \eqref{1.14} satisfying the estimate
\begin{equation} \label{2.7}
e^{\lambda t} \mathcal{E}(t) + \int_{0}^{t} \mathcal{D}(s)ds \lesssim \mathcal{E}(0).
\end{equation}

\end{theo}
\begin{rema}
Since $\eta$ is such that the mapping $\Phi(\cdot,t)$, defined by \eqref{1.4}, is a diffeomorphism for each $t\geq0$, one may change coordinate to $y\in\Omega(t)$ to produce a global-in-time decaying solution to \eqref{1.2}.

\end{rema}

\begin{rema}
Theorem \ref{thm:0102} implies that $\mathcal E(t)\lesssim e^{-\lambda t}$, which means that for $\sigma>0$ the solution returns to the stable state at an exponential decay rate.
\end{rema}
We then state our results for \eqref{1.14} in the case $\sigma=0$. And we first define some energy functionals corresponding to this case. For a generic integer $n\geq3$, we define the energy as
\begin{equation}\label{2.1}
\begin{aligned}
\mathcal{E}_{n} :=\sum_{j=0}^{n}\left ( \l\partial_{t}^{j} u\r^{2}_{2n-2j} + \l\partial_{t}^{j} b \r^{2}_{2n-2j}+ \l\partial_{t}^{j} \eta\r^{2}_{2n-2j} \right)+ \sum_{j=0}^{n-1}\l\partial_{t}^{j}p\r^{2}_{2n-2j-1},
\end{aligned}
\end{equation}
and define the corresponding dissipation as
\begin{equation}\label{2.2}
\begin{aligned}
\mathcal{D}_{n} &:=\sum_{j=0}^{n} \left( \l\partial_{t}^{j} u\r^{2}_{2n-2j+1} + \l\partial_{t}^{j}b\r^{2}_{2n-2j+1}\right)+ \sum_{j=0}^{n-1}\l\partial_{t}^{j}p\r^{2}_{2n-2j}\\
&+ \l\eta\r_{2n-1/2}^{2}+ \l\partial_{t} \eta\r^{2}_{2n-1/2} + \sum_{j=2}^{n+1} \l\partial_{t}^{j} \eta\r^{2}_{2n-2j+5/2}.
\end{aligned}
\end{equation}
We write the high-order spatial derivatives of $\eta$ as
\begin{equation} \label{2.3}
\mathcal{F}_{2N} := \|\eta\|^{2}_{4N+1/2}.
\end{equation}

Finally, we define the total energy as
\begin{equation}\label{2.5}
\begin{aligned}
\mathcal{G}_{2N}(t) := \sup_{0\leq r \leq t} \mathcal{E}_{2N}(r)+ \int_{0}^{t} \mathcal{D}_{2N}(r) dr+ \sup_{0\leq r\leq t }(1+r)^{4N-8} \mathcal{E}_{N+2}(r) + \sup_{0\leq r\leq t} \frac{\mathcal{F}_{2N}(r)}{(1+r)}.
\end{aligned}
\end{equation}

Our main results state as follows.

\begin{theo}\label{theo2.1}
 For $\sigma=0$, we assume that the initial data $u_{0}\in H^{4N}(\Omega)$, $b_{0}\in H^{4N}(\Omega)$ and $\eta_{0}\in H^{4N+1/2}(\Sigma)$ satisfy some appropriate compatibility conditions as well as the zero-average condition \eqref{1.3}, where $N\geq 3$. There exists a constant $\varepsilon_0>0$ such that if
\begin{equation*}
\mathcal{E}_{2N}(0)+\mathcal{F}_{2N}(0)\leq \varepsilon_0,
\end{equation*}
then, for all $t\geq0$, there exists a global unique solution $(u,p,b,\eta)$ to \eqref{1.14} satisfying the estimate
\begin{equation}\label{2.8}
\mathcal{G}_{2N}(t)\lesssim \mathcal{E}_{2N}(0)+\mathcal{F}_{2N}(0).
\end{equation}

\end{theo}
\begin{rema}
Theorem \ref{theo2.1} implies that $\mathcal E_{N+2}(t)\lesssim (1+t)^{-{4N-8}}$, which is integrable in time for $N\geq 3$. Since $N$ may be taken to be arbitrarily large, this decay results can be regarded as an ``almost exponential'' decay rate. Comparing the two different cases for $\sigma$, reveals that the surface tension plays a important role for the decay rate.
\end{rema}

\begin{rema}
We refe to \cite{Bealee,Guo} for the local well-posedness of the system \eqref{1.14} for both the case $\sigma>0$ and $\sigma=0$, respectively. Then, by a continuity argument, to prove Theorem \ref{thm:0102} and Theorem \ref{theo2.1} it suffices to derive the a priori estimates, namely, Theorem 4.6 and 5.12.
\end{rema}
\section{Preliminaries for a priori estimates}
In this section, we will present some preliminary results and given the proofs respectively. We state two forms of equations to \eqref{1.14} and describe the corresponding energy evolution structure.
\subsection{Geometric Form}
We now give a linear formation of the problem \eqref{1.14} in its geometric form. Assume that $u,~\eta,~b$ are known and that $\ma,~\mathcal{N},~J,$ etc., are given in terms of $\eta$ as usual. We then consider the linear equation for $(v,~H,~q,~h)$ given by
\begin{equation}\label{3.1}
\left\{
\begin{aligned}
&\partial_{t} v-\partial_{t}\bar{\eta}\tilde{b}K \partial_{3} v + u\cdot \nabla_{\mathcal{A}} v
+{\rm div}_\ma(qI-\mathbb{D}_{\mathcal{A}}v)=(b+\bar B)\cdot\nabla_{\ma}H+F^{1},~~{\rm in~}\Omega~ \\
&{\rm div}_{\mathcal{A}} v=F^{2},~~~~~~~~~~~~~~~~~~~~~~~~~~~~~~~~~~~~~~~~~~~~~~~~~~~~~~~~~~~~~~~~~~~~~~~~~~~~{\rm in}~\Omega~ \\
&\partial_tH-\partial_{t}\bar{\eta}\tilde{b}K \partial_{3} H + u\cdot \nabla_{\mathcal{A}} H -\Delta_{\mathcal{A}} H=(b+\bar B)\cdot\nabla_{\ma}v+F^{3},~~~~~~~~~~~~~~{\rm in}~\Omega~ \\
&(qI-\mathbb{D}_{\mathcal{A}} v)\mathcal{N}=(h - \sigma \Delta_{\star}h) \mathcal{N}+F^{4},~H=0,~~~~~~~~~~~~~~~~~~~~~~~~~~~~~~~~~~{\rm on}~\Sigma~\\
&\partial_{t} h-v\cdot \mathcal{N} =F^{5}, ~~~~~~~~~~~~~~~~~~~~~~~~~~~~~~~~~~~~~~~~~~~~~~~~~~~~~~~~~~~~~~~~~~~~~{\rm on}~\Sigma~\\
&v=0,~H=0,~~~~~~~~~~~~~~~~~~~~~~~~~~~~~~~~~~~~~~~~~~~~~~~~~~~~~~~~~~~~~~~~~~~~~~~~~~{\rm on}~\Sigma_{-1},~
\end{aligned}
\right.
\end{equation}
where $\Delta_\star=\partial_{x_1}^2+\partial_{x_2}^2$.
\begin{lemm}\label{lemm3.1}
Let u and $\eta$ be given and solve $\eqref{1.14}$. If $(v,H,q,h)$ solve $\eqref{3.1}$ then
\begin{equation}\label{3.2}
\begin{aligned}
&\frac{d}{dt}\left(\int_{\Omega} \frac{|v|^{2}}{2} J+ \int_{\Omega} \frac{|H|^{2}}{2} J+\int_{\Sigma} \frac{|h|^{2}}{2}
+\sigma\int_{\Sigma} \frac{|Dh|^{2}}{2}\right)+ \int_{\Omega} \frac{|\mathbb{D}_{\mathcal{A}} v|^{2}}{2} J +\int_{\Omega}|\nabla_\ma H|^2J\\
&~~~~~~~= \int_{\Omega} (v\cdot F^{1}+ q F^{2}+v\cdot F^3)J
- \int_{\Sigma} v\cdot F^{4} +\int_{\Sigma}( h-\sigma\Delta_{\star}h) F^{5} ,
\end{aligned}
\end{equation}
\end{lemm}
\begin{proof}
We take the inner product of the first equation in \eqref{3.1} with $Jv$ and the third equation with $JH$, then integrate over $\Omega$ to find that
\begin{equation*}
I_1 + I_2+I_3+I_4=I_5,
\end{equation*}
where
\begin{equation}
\begin{aligned}
I_1=&\int_{\Omega} (\partial_{t} v_{i} J v_{i} -\partial_{t}\bar{\eta} \tilde{b} \partial_{3} v_{i}v_{i}+ u_{j}\mathcal{A}_{jk}\partial_{k}v_{i} J v_{i}),\\
I_2=&\int_{\Omega} \ma_{ik}\partial_kq Jv_{i}-\int_{\Omega} \ma_{jk}\partial_k(\ma_{jl}\partial_l v_i+\ma_{il}\partial_l v_j) J v_{i},\\
I_3=&\int_{\Omega}(\partial_{t} H_{i} J H_{i} -\partial_{t}\bar{\eta} \tilde{b} \partial_{3} H_{i}H_{i}+ u_{j}\mathcal{A}_{jk}\partial_{k}H_{i} J H_{i})\\
&-\int_{\Omega} \ma_{jk}\partial_k(\ma_{jl}\partial_l H_i)J H_{i},\\
I_4=&\int_\Omega(\bar B_j+b_j)\ma_{jk}\partial_k H_i Jv_i+
\int_\Omega(\bar B_j+b_j)\ma_{jk}\partial_k v_i JH_i,\\
I_5=& \int_{\Omega}F_i^{1}J v_i
+ \int_{\Omega}F_i^{3}J H_i.\\
\end{aligned}
\end{equation}
Integrating by parts and using \eqref{1.13}, one has
\begin{equation}
\begin{aligned}
I_1=&\partial_t\int_{\Omega} \frac{|v|^{2} J}{2}-\int_\Omega\frac{|v|^2\partial_tJ}{2}
-\int_\Omega\partial_t\bar\eta\tilde b\partial_3\frac{|v|^2}{2}+\int_\Omega u_j\partial_k(J\ma_{jk}\frac{|v|^2}{2})\\
=&\partial_t\int_{\Omega} \frac{|v|^{2} J}{2}-\int_\Omega\frac{|v|^2\partial_tJ}{2}+\int_{\Omega}\frac{|v|^2}{2}(\partial_t\bar\eta+\tilde b\partial_t\partial_3\bar\eta)\\
&-\int_\Omega J\ma_{jk}\partial_ku_j\frac{|v|^2}{2}-\frac{1}{2}\int_\Sigma(\partial_t\eta|v|^2-u_jJ\ma_{jk}e_3\cdot e_k|v|^2)\\
=&\partial_t\int_{\Omega} \frac{|v|^{2} J}{2},
\end{aligned}
\end{equation}
where according to \eqref{1.11}, we know that $\partial_tJ=\partial_t\bar \eta+\tilde b\partial_t\partial_3\bar\eta$ and $J\ma_{jk}e_3\cdot e_k=\mathcal{N}_j$ on $\Sigma$, then use the condition $\partial\eta=u\cdot \mathcal{N}$. Similarly, an integration by parts reveals that
\begin{equation}
\begin{aligned}
I_2&= -\int_{\Omega} \mathcal{A}_{jk} (qI-\mathbb{D}_\ma v)_{ij} J \partial_{k} v_{i} + \int_{\Sigma} J \mathcal{A}_{j3} (qI-\mathbb{D}_\ma v) _{ij} v_{i}\\
& = \int_{\Omega}( -q \mathcal{A}_{ik} \partial_{k} v_{i} J + J\frac{|\mathbb{D}_{\mathcal{A}} v|^{2}}{2})+ \int_{\Sigma} (qI-\mathbb{D}_\ma v)_{ij} \mathcal{N}_{j} v_{i} \\
& = \int_{\Omega}( -q J F^{2} + J\frac{|\mathbb{D}_{\mathcal{A}} v|^{2}}{2})+ \int_{\Sigma}( h-\sigma\Delta_{\star}h) \mathcal{N} \cdot v + F^{4}\cdot v\\
&=\int_{\Omega}( -q J F^{2} + J\frac{|\mathbb{D}_{\mathcal{A}} v|^{2}}{2})+\int_{\Sigma}( h-\sigma\Delta_{\star}h)(\partial_t h-F^5)+ F^{4}\cdot v\\
&=\int_{\Omega} (-q J F^{2} + J\frac{|\mathbb{D}_{\mathcal{A}} v|^{2}}{2})+\partial_t\int_\Sigma(\frac{|h|^2}{2}+\sigma|Dh|^2)\\
&~~~~+\int_\Sigma v\cdot F^4-\int_\Sigma(h-\sigma\Delta_\star h)\cdot F^5.
\end{aligned}
\end{equation}
By using $H=0$ on $\partial\Omega$, ${\rm div}_{\ma} u=0$ and \eqref{1.13}, one has
\begin{equation}
\begin{aligned}
I_3=&\partial_t\int_{\Omega} \frac{|H|^{2} J}{2}-\int_\Omega\frac{|H|^2\partial_tJ}{2}
+\int_\Omega(\partial_t\bar\eta+\partial_3\partial_t\bar\eta\tilde b)\frac{|H|^2}{2}\\
&-\int_\Omega J\ma_{jk}\partial_k u_j
+\int_\Omega J|\nabla_\ma H|^2\\
=&\partial_t\int_{\Omega} \frac{|H|^{2} J}{2}+\int_\Omega J|\nabla_\ma H|^2,
\end{aligned}
\end{equation}
and, similarly, by using $H=0$ on $\partial\Omega$, ${\rm div}_{\ma} b=0$ and \eqref{1.13}, we deduce
\begin{equation}
I_4=\int_\Omega(\bar B_j+b_j)J\ma_{jk}\partial_k(H_iv_i)
=-\int_\Omega J\ma_{jk}\partial_k b_j H_iv_i=0.
\end{equation}
Then, \eqref{3.2} follows from the estimates of $I_1,~I_2,~I_3$ and $I_4$.
\end{proof}

\subsection{Perturbed Linear Form}
 In many parts of this paper we will apply the PDE in a different formulation, which looks like a perturbation of the
 linearized problem. The utility of this form of the equations lies in the fact that the linear operator have constant coefficients. The equations in this form are
\begin{equation}\label{3.13}
\left\{
\begin{aligned}
&\partial_{t} u + \nabla p-\Delta u=G^{1},~~~~~~~~~~~~~~~~~~~~~~~~~~~~~~~~~~{\rm in}~\Omega~ \\
&{\rm div} u=G^{2},~~~~~~~~~~~~~~~~~~~~~~~~~~~~~~~~~~~~~~~~~~~~~~~~{\rm in}~\Omega~ \\
&\partial_{t}b -\Delta b=G^{3},~~~~~~~~~~~~~~~~~~~~~~~~~~~~~~~~~~~~~~~~~~~{\rm in}~\Omega~ \\
&(pI-\mathbb{D} u)e_{3}=(\eta -\sigma \Delta_\star \eta)e_3+G^{4},~b=0,~~~~~~~~{\rm on}~\Sigma~\\
&\partial_{t} \eta- u_{3} =G^{5}, ~~~~~~~~~~~~~~~~~~~~~~~~~~~~~~~~~~~~~~~~~~~~{\rm on}~\Sigma\\
&u=0,~b=0,~~~~~~~~~~~~~~~~~~~~~~~~~~~~~~~~~~~~~~~~~~~~~~{\rm on}~\Sigma_{-1}.
\end{aligned}
\right.
\end{equation}
Here we have written the nonlinear terms $G^{i}$ for $i=1,...,5$ as follows. We write $G^{1,l}:= \Sigma ^5_{l=1}G^{1,l}$, for
\begin{equation}\label{3.11}
\begin{aligned}
G_{i}^{1,1}: = & (\delta_{ij}-\mathcal{A}_{ij})\partial_{j} p,~G_{i}^{1,2}: = \partial_{t} \bar{\eta} \tilde b K \partial_{3} u_{i},\\
G_{i}^{1,3}: =& -u_{j} \mathcal{A}_{jk}\partial_{k} u_{i}+(b_j+\bar B_j)\ma_{jk}\partial_kb_i, \\
G_{i}^{1,4}: =& [K^{2}(1+A^{2}+B^{2})-1]\partial_{33} u_{i}-2AK \partial_{13} u_{i} -2BK \partial_{23}u_{i} , \\
G_{i}^{1,5}: =& [-K^{3} (1+A^{2}+B^{2})\partial_{3}J +AK^{2} (\partial_{1} J+ \partial_{3} A)]\partial_{3} u_{i}\\
&+ [BK^{2}(\partial_{2} J+ \partial_{3} B)-K(\partial_1A+\partial_2B)]\partial_{3} u_{i},\\
\end{aligned}
\end{equation}
$G^2$ is the function
\begin{equation}
G^{2} := AK \partial_{3} u_{1} + BK \partial_{3} u_{2} + (1-K) \partial_{3} u_{3} ,
\end{equation}
and
$G^{3}=G^{3,1}+G^{3,2}+G^{3,3}+G^{3,4}$, for
\begin{equation}
\begin{aligned}
G_i^{3,1}:=&\partial_{t} \bar{\eta} \tilde b K \partial_{3} b_{i}\\
G_{i}^{3,2}: =& -u_{j} \ma_{jk}\partial_{k} b_i+(b_j+\bar B_j)\ma_{jk}\partial_ku_i, \\
G_{i}^{3,3}: =& [K^{2}(1+A^{2}+B^{2})-1]\partial_{33} b_i-2AK \partial_{13} b_{i} -2BK \partial_{23}b_i , \\
G_{i}^{3,4}: = &[-K^{3} (1+A^{2}+B^{2})\partial_{3}J +AK^{2} (\partial_{1} J+ \partial_{3} A)]\partial_{3} b_i\\
&+ [BK^{2}(\partial_{2} J+ \partial_{3} B)-K(\partial_1A+\partial_2B)]\partial_{3} u_1,
\end{aligned}
\end{equation}
\begin{eqnarray}
&&G^{4}: = \partial_{1} \eta \left(
\begin{aligned}
&~~~~p-\eta-2(\partial_{1} u_{1}-AK \partial_{3} u_{1}) \\
&-\partial_{2} u_{1} -\partial_{1} u_{2} + BK \partial_{3} u_{1} + AK \partial_{3} u_{2}  \\
&~~~-\partial_{1} u_{3} -K \partial_{3} u_{1} + AK\partial_{3} u_{3}
\end{aligned}
\right) ~\nonumber~\\
&&+ \partial_{2} \eta \left(
\begin{aligned}
&-\partial_{2} u_{1} -\partial_{1} u_{2} + BK \partial_{3} u_{1} + AK \partial_{3} u_{2} \\
&~~~~~~p-\eta-2(\partial_{2} u_{2}-BK \partial_{3} u_{2})   \\
&~~~~~~-\partial_{2} u_{3} -K \partial_{3} u_{2} + BK\partial_{3} u_{3}
\end{aligned}
\right) +
\left(
\begin{aligned}
&(K-1) \partial_{3} u_{1}+ AK \partial_{3} u_{3}  \\
&(K-1) \partial_{3} u_{2}+ BK \partial_{3} u_{3}    \\
&~~~~~~~2(K-1)\partial_{3} u_{3}
\end{aligned}
\right)~\nonumber~\\
&&+\sigma(H-\Delta_{\star}\eta)\mathcal{N}+ \sigma \Delta_{\star}(\mathcal{N}-e_{3}),
\end{eqnarray}

\begin{equation}\label{3.12}
G^5=-D\eta\cdot u.
\end{equation}

\begin{lemm}\label{lemm3.2}
Suppose $(v,H,q,h)$ solve
\begin{equation}\label{3.14}
\left\{
\begin{aligned}
&\partial_{t} v + \nabla q-\Delta v =\Phi^{1},~~~~~~~~~~~~~~~~~~~~~~~~~~~~~~~{\rm in}~\Omega~ \\
&{\rm div} v=\Phi^{2},~~~~~~~~~~~~~~~~~~~~~~~~~~~~~~~~~~~~~~~~~~~~{\rm in}~\Omega~ \\
&\partial_{t}H -\Delta H=\Phi^{3},~~~~~~~~~~~~~~~~~~~~~~~~~~~~~~~~~~~~{\rm in}~\Omega~ \\
&(qI-\mathbb{D} v)e_{3}=(h -\sigma\Delta_{\star}h)e_3+\Phi^{4},~H=0,~~~~~{\rm on}~\Sigma~\\
&\partial_{t} h- v_{3} =\Phi^{5}, ~~~~~~~~~~~~~~~~~~~~~~~~~~~~~~~~~~~~~~~{\rm on}~\Sigma\\
&v=H=0,~~~~~~~~~~~~~~~~~~~~~~~~~~~~~~~~~~~~~~~~~~~{\rm on}~\Sigma_{-1}.
\end{aligned}
\right.
\end{equation}
Then
\begin{equation}\label{3.15}
\begin{aligned}
\partial_t&\left(\int_{\Omega} \frac{|v|^{2}}{2} +\int_{\Omega}\frac{|H|^2}{2}+ \int_{\Sigma} \frac{|h|^{2}}{2}
+\sigma\int_\Sigma\frac{| Dh|^2}{2}\right)+ \int_{\Omega} \frac{|\mathbb{D} v|^{2}}{2}+ \int_\Omega|\nabla H|^2\\
= &\int_{\Omega} v\cdot( \Phi^{1}-\nabla\Phi^2)+\int_{\Omega}( q  \Phi^{2}+ H\cdot \Phi^3)-  \int_{\Sigma} v\cdot \Phi^{4} + \int_{\Sigma}  (h-\sigma\Delta_\star h) \Phi^{5},
\end{aligned}
\end{equation}

\end{lemm}
\begin{proof}
From the first and second equation in \eqref{3.14}, we can rewrite the first one as
\begin{equation}\label{3.16}
\partial_tv+{\rm div}(qI-\mathbb{D}v)=\Phi^1-\nabla\Phi^2.
\end{equation}
Taking the inner product of the \eqref{3.16} with $v$ and the third equation in \eqref{3.14} with $H$, integrating by parts over $\Omega$ and then adding the resulting equations together, one has
\begin{equation*}
\begin{aligned}
&\partial_t \left(\int_{\Omega} \frac{|v|^{2}}{2}+\int_{\Omega} \frac{|H|^{2}}{2}+\int_\Sigma \frac{|h|^2}{2}
+\sigma\int_\Sigma\frac{| Dh|^2}{2}\right)
-\int_\Omega q{\rm div} v+\int_\Omega\frac{|\mathbb{D}v|^2}{2}\\
+\int_\Sigma&(qI-\mathbb{D}v)e_3\cdot v+\int_\Omega|\nabla H|^2=\int_\Omega(\Phi^1-\nabla\Phi^2)\cdot v+\int_\Omega\Phi^3\cdot H.
\end{aligned}
\end{equation*}
Furthermore, we bring ${\rm div}v=\Phi ^2$, $(qI-\mathbb{D}v)e_3=(h-\sigma\Delta_\star h)e_3+\Phi^4$ and $v_3=\partial_t h-\Phi^5$ into the above equation, then \eqref{3.15} follows.
\end{proof}

\subsection{Some useful estimates}

Before having a priori estimates on the nonlinear terms, we give the useful $L^\infty$ estimates for removing the appearance of $J$ factors.
\begin{lemm}\label{lemm4.1}
There exists a universal $0<\delta<1$ so that if $\|\eta\|^{2}_{5/2}\leq \delta$, then we have the estimate
\begin{equation} \label{4.1}
\| J-1\|^{2}_{L^{\infty}} + \|A\|_{L^{\infty}}^{2}+ \|B\|_{L^{\infty}}^{2}\leq \frac{1}{2}, ~~and~~\|K\|^{2}_{L^{\infty}}+\|\mathcal{A} \|^{2}_{L^{\infty}}\lesssim 1,
\end{equation}
\end{lemm}
\begin{proof}
According to the definitions of $A,~B,~J$ given in \eqref{1.10}-\eqref{1.11} and Lemma \ref{lemmA.4}, we have that
 \begin{equation}
 \|J-1\|_{L^{\infty}}^2+\|A\|^2_{L^\infty}+\|B\|^2_{L^\infty}\lesssim\|\bar \eta\|^2_3\lesssim\|\eta\|^2_{5/2}.
 \end{equation}
 Then if $\delta$ is sufficiently small, \eqref{4.1} follows.
\end{proof}
Furthermore, we provide an estimate for $\partial^n_t\ma$.
\begin{lemm}\label{lemm4.2}
For $n=2N$ or $n=N+2$, we have
\begin{equation}
\l\partial^{n+1}_t J\r^2_0+\l\partial^{n+1}_t\ma\r^2_0\lesssim\mathcal D_n.
\end{equation}
\end{lemm}
\begin{proof}
 Since temporal derivatives commute with the Poisson integral, applying Lemma \ref{lemmA.4}, we have
\begin{equation*}
\l\partial^{m+1}_t\bar\eta\r^2_1=\l\partial^{m+1}_t\bar\eta\r^2_0+\l\nabla\partial^{m+1}_t\bar\eta\r^2_0\lesssim
\l\partial^{m+1}_t \eta\r^2_{1/2},~~{\rm for}~m\geq0.
\end{equation*}
From the definition of $\mathcal D_n$, we have
\begin{equation}
\l\partial^{n+1}_t\eta\r^2_{1/2}\lesssim\mathcal D_n,~~{\rm for}~n=2N~ {\rm or~}n=N+2.
\end{equation}
Then, according to the definition of $J,~A,~B$ and $K$, we have
\begin{equation*}
\l\partial_t^{n+1}J\r^2_0+\l\partial_t^{n+1}A\r^2_0+\l\partial_t^{n+1}B\r^2_0+\l\partial_t^{n+1} K\r\lesssim \mathcal D_n,~{\rm for}~n=2N~ {\rm or~}n=N+2.
\end{equation*}
Using the Sobolev embeddings we complete the proof of $\ma$ since the components of $\ma$ are either unity, $K,~AK$ or $BK$.
\end{proof}

\section{For the case $\sigma>0$}

\subsection{Nonlinear estimates}

We will employ the form \eqref{3.1} to study the temporal derivative of solutions to \eqref{1.14}. That is, we employ $\partial_{t}$ to \eqref{1.14} and set $(v,H,q,h)=(\partial_{t}u,\partial_{t} b,\partial_{t}p,\partial_{t}\eta)$ satisfying \eqref{3.1} for certain terms $F^{i}$. Below we record the form of these forcing terms $F^{i}, i=1,2,3,4,5$.
\begin{equation}\label{13.4}
\begin{aligned}
&F^1_i=\partial_{t} (\partial_{t} \overline{\eta} \tilde{b} K) \partial_{3}u-\partial_{t} (u_{j} \mathcal{A}_{jk}) \partial_{k} u_{i}-\partial_{t} \mathcal{A}_{ik} \partial_{k} p+\partial_{t} \mathcal{A}_{jk} \partial_{k}(\mathcal{A}_{im}\partial_{m} u_{j}+ \mathcal{A}_{jm} \partial_{m} u_{i})\\
&~~~~~~~~~~~~~~~+\mathcal{A}_{jk} \partial_{k} (\partial_{t} \mathcal{A}_{im} \partial_{m} u_{j}+ \partial_{t} \mathcal{A}_{jm} \partial_{m} u_{i})+\partial_t((b_j+\bar B_j)\ma_{jk})\partial_k b_i,
\end{aligned}
\end{equation}
\begin{equation}\label{13.5}
F^{2}_i = -\partial_{t} \mathcal{A}_{ij} \partial_{j} u_{i},
\end{equation}
\begin{equation}\label{13.6}
\begin{aligned}
&F^{3}= \partial_{t} (\partial_{t} \overline{\eta} \tilde{b} K) \partial_{3} b-\partial_{t} (u_{j} \mathcal{A}_{jk})\partial_{k}b
+ \partial_{t} \mathcal{A}_{il} \partial_{l}\mathcal{A}_{im}\partial_{m} b\\
&~~~~~~~~~~~~+ \mathcal{A}_{il} \partial_{l}\partial_{t} \mathcal{A}_{im}\partial_{m}b+\partial_t((b_j+\bar B_j)\ma_{jk})\partial_k b,
\end{aligned}
\end{equation}
\begin{equation}\label{13.7}
\begin{aligned}
&F_{i}^{4} =  (\mathcal{A}_{ik} \partial_{k} u_{j}+ \mathcal{A}_{jk} \partial_{k} u_{i} )\partial_{t} \mathcal{N}_{j} +(\partial_{t}\mathcal{A}_{ik} \partial_{k} u_{j}+ \partial_{t}\mathcal{A}_{jk} \partial_{k} u_{i} ) \mathcal{N}_{j}\\
&~~~~~~~~~~~+(\eta-p)\partial_{t} \mathcal{N}_{i} -(\sigma \partial_{t} M -\sigma\partial_{t} \Delta_{\star} \eta) \mathcal{N}_{i} -\sigma M \partial_{t} \mathcal{N}_{i},
\end{aligned}
\end{equation}
\begin{equation}\label{13.9}
F^{5} = \partial_{t} D \eta \cdot u,
\end{equation}

Next, we will estimate the nonlinear terms $F^{i}$ for $i=1,...,6$, which will be used principally to estimates the interaction terms on the right side of \eqref{3.2}.

\begin{lemm}\label{lemm14.1}
Let $F^{1},...,F^{5}$  de defined in \eqref{13.4}-\eqref{13.9}. let $\mathcal{E}$ and $\mathcal{D}$ be as defined in \eqref{12.1} and \eqref{12.2}. Suppose that $\mathcal{E}\leq \delta$, where $\delta\in (0,1)$ is the universal constant given in Lemma \ref{lemm4.1}. Then,
\begin{equation}\label{14.1}
\|F^{1}\|_{0}+\|F^{2}\|_{0}+\|F^{3}\|_{0}+\|F^{4}\|_{0}+ \|F^{5}\|_{0}\lesssim \sqrt{\mathcal{E}\mathcal{D}},
\end{equation}
\begin{equation}\label{14.2}
\left|\int_{\Omega}  p\partial_t( F^{2} J)\right| \lesssim \sqrt{\mathcal{E}} \mathcal{D},
~~~and~~~\left|\int_{\Omega} p F^{2} J\right| \lesssim \mathcal{E}^{\frac{3}{2}}.
\end{equation}
\end{lemm}
\begin{proof}
Throughout the lemmas we will employ Holder's inequality, Sobolev embeddings, trace theory, Lemma \ref{lemm4.1} and Lemma \ref{lemmA.4}.
Firstly, we give the estimates for $F^{1}$.
\begin{equation*}
\begin{aligned}
\|\partial_t(\partial{\bar\eta}\tilde bK)\partial_3u\|_{0}\lesssim&\|\partial^2_t\bar\eta\|_0\|K\|_{L^\infty}\|\partial_3u\|_{L^\infty}
+\|\partial_t \bar{\eta}\|_{L^\infty}\|\partial_t \bar\eta\|_{1}\|\partial_3u\|_{L^\infty}\\
\lesssim&  \|\partial^2_t\eta\|_{-1/2}\|u\|_{3}+\|\partial_t\eta\|_{3/2}\|\partial_t\eta\|_{1/2}\|u\|_3\\
\lesssim&(\sqrt\mathcal E+\mathcal E)\sqrt D\lesssim \sqrt{\mathcal{E}\mathcal{D}}.
\end{aligned}
\end{equation*}
and the other terms of $F^1$ can be bounded in a similar way. Next, we control the second term $F^{2}$ as follows
\begin{equation*}
\|F^{2}\|_{0}\lesssim \|\partial_{t} \nabla \overline{\eta}\|_{0} \|\nabla u\|_{L^{\infty}} \lesssim\sqrt{\mathcal{E}\mathcal{D}} ,
\end{equation*}
Similar to the $F^{1},F^{2}$ term, whereas $F^{3},F^{4},F^{5}$ term can be handled as follows
\begin{equation*}
\|F^{3}\|_{0}+\|F^{4}\|_{0}+ \|F^{5}\|_{0}\lesssim \sqrt{\mathcal{E}\mathcal{D}}.
\end{equation*}
For the term involves in \eqref{14.2}, we have
\begin{equation*}
\begin{aligned}
&\left|\int_\Omega p(\partial_tJF^2+J\partial_tF^2)\right|\\
\lesssim& \|p\|_{L^\infty}\|\partial_t\bar\eta\|_1\|F^2\|_0
+ \|p\|_{L^\infty}\|J\|_{L^\infty}(\|u\|_1\|\partial^2_t\bar\eta\|_1
+\|\partial_t\bar\eta\|_1\|\partial_tu\|_1)\\
\lesssim&\|p\|_2\|\partial_t\eta\|_{1/2}\|F^2\|_0+\|p\|_2
(\|u\|_1\|\partial^2_t\eta\|_{1/2}+\|\partial_t\eta\|_{1/2}\|\partial_tu\|_1)\\
\lesssim&(\mathcal E\mathcal D+\sqrt\mathcal E\mathcal D)\lesssim\sqrt\mathcal E\mathcal D,
\end{aligned}
\end{equation*}
and
\begin{equation*}
\begin{aligned}
\left|\int_\Omega pJF^2\right|\lesssim\|p\|_{L^6}\|\partial_t\nabla\bar\eta\|_{L^{2}}\|\nabla u\|_{L^3}\|J\|_{L^\infty}
\lesssim\|p\|_1\|u\|_2\|\partial_t\eta\|_{1/2}\lesssim\mathcal E^{3/2}.
\end{aligned}
\end{equation*}
Then we complete the proof of this lemma.
\end{proof}

Then, we turn our attention to the nonlinear terms $G^{i}$ for $i=1,...,5$, as defined in \eqref{3.11}-\eqref{3.16}.

\begin{lemm}\label{lemm14.2}
Let $G^{1},...,G^{5}$  de defined in \eqref{3.11}-\eqref{3.16} and let $\mathcal{E}$ and $\mathcal{D}$ be as defined in \eqref{12.1} and \eqref{12.2}. Suppose that $\mathcal{E}\leq \delta$, where $\delta\in (0,1)$ is the universal constant given in Lemma \ref{lemm4.1}, and that $D<\infty$. Then,
\begin{equation}\label{14.3}
\l G^{1}\r_{1} +\l G^{2}\r_{2}+ \l G^{3} \r_{1} + \l G^{4}\r_{3/2}+\l G^{5}\r_{5/2}+\l\partial_tG^{5}\r_{1/2} \lesssim \sqrt{\mathcal{E}\mathcal{D}},
\end{equation}
and
\begin{equation}\label{14.4}
\l G^{1}\r_{0} +\l G^{2}\r_{1}+\l G^{2}\r_{-1} +\l G^{3} \r_{0} + \l  G^{4}\r_{1/2}+\l G^{5}\r_{3/2}+\l G^{5}\r_{-1/2}\lesssim \mathcal{E}.
\end{equation}
\end{lemm}
\begin{proof}
Here the estimates of $G^{1},...,G^{5}$  similar as [\cite{Kim}, Theorem 4.3] , so we omit it.
\end{proof}
 
\subsection{A priori estimates}
In this section we combine energy-dissipation estimates with various elliptic estimates and estimate the nonlinearities in order to deduce a system of a priori estimates.

\subsubsection{Energy-dissipation estimates}
In order to state our energy-dissipation estimates we must first introduce some notation. Recall that for a multi-index $\alpha=(\alpha_{0},\alpha_{1},\alpha_{2})\in \mathbb{N}^{1+2}$ we write $|\alpha|=2\alpha_{0}+ \alpha_{1}+\alpha_{2}$ and $\partial^{\alpha}=\partial_{t}^{\alpha_{0}}\partial^{\alpha_{1}}_{1}
\partial^{\alpha_{2}}_{2}$. For $\alpha \in \mathbb{N}^{1+2}$ we set
\begin{eqnarray}\label{15.1}
&& \overline{\mathcal{E}}_{\alpha} := \int_{\Omega} \frac{1}{2} |\partial^{\alpha} u|^{2} + \int_{\Sigma} (\frac{1}{2}
|\partial^{\alpha}\eta|^{2} + \frac{\sigma}{2} |D \partial^{\alpha} \eta|^{2}) + \int_{\Omega} \frac{1}{2} |\partial^{\alpha} b|^{2},  ~\nonumber~\\
&&\overline{\mathcal{D}}_{\alpha} := \int_{\Omega} \frac{1}{2} |\mathbb{D} \partial^{\alpha} u|^{2}  + \int_{\Omega} |\nabla \partial^{\alpha} b|^{2}.
\end{eqnarray}
We then define
\begin{equation}\label{15.2}
\overline{\mathcal{E}}:= \sum_{|\alpha|\leq 2} \overline{\mathcal{E}}_{\alpha}~~~and~~~\overline{\mathcal{D}}:= \sum_{|\alpha|\leq 2} \overline{\mathcal{D}}_{\alpha}.
\end{equation}
We will also need to use the functional
\begin{equation}\label{15.3}
\mathcal{F} :=\int_{\Omega} p F^{2} J.
\end{equation}

Our next result encodes the energy-dissipation inequality associated to $\overline{\mathcal{E}}$ and $\overline{\mathcal{D}}$.

\begin{lemm}\label{lemm4.3}
Suppose that $(u,b,p,\eta)$ solves \eqref{1.14}. Let $\mathcal{E}$ and $\mathcal{D}$ defined in \eqref{12.1} and \eqref{12.2}. Assume that $\mathcal{E}\leq \delta$, where $\delta\in(0,1)$ is the universal constant given in Lemma \ref{lemm4.1}.
Let $\overline{\mathcal{E}}$ and $\overline{\mathcal{D}}$ be given by \eqref{15.2} and $\mathcal{F}$ be given by \eqref{15.3}. Then
\begin{equation}\label{15.4}
\frac{d}{dt} (\overline{\mathcal{E}}-\mathcal{F}) + \overline{\mathcal{D}}\lesssim \sqrt{\mathcal{E}} \mathcal{D},
\end{equation}
for all $t\in[0,T]$.
\end{lemm}
\begin{proof}
Let $\alpha\in \mathbb{N}^{1+2}$ with $|\alpha|\leq 2$. We apply $\partial^{\alpha}$ to \eqref{1.14} to derive an equation for $(\partial^{\alpha} u, \partial^{\alpha}b, \partial^{\alpha}\eta, \partial^{\alpha} p)$. We will consider the form of this equation in different ways depending on $\alpha$.

Suppose that $\alpha=(1,0,0)$, i.e. that $\partial^{\alpha}=\partial_{t}$. Then $v=\partial_t u$, $q=\partial_tp$, $H=\partial_tb$, $h=\partial_t\eta$ satisfying \eqref{3.1} with $F^{1},...,F^{5}$ defined in \eqref{13.4}-\eqref{13.9}. Then according to Lemma \ref{lemm3.1} and Lemma \ref{lemm14.1}, we deduce
\begin{eqnarray*}
&& \frac{d}{dt} (\int_{\Omega} \frac{|\partial_{t} u|^{2}}{2}J + \int_{\Omega} \frac{|\partial_{t} b|^{2}}{2}J+ \int_{\Sigma} \frac{|\partial_{t}\eta|^{2}}{2}+ \int_{\Sigma} \sigma \frac{|D \partial_{t} \eta|^{2}}{2} )+ \mu\int_{\Omega} \frac{|\mathbb{D}_{\mathcal{A}} \partial_{t} u|^{2}}{2}J+\kappa \int_{\Omega} |\nabla_{\mathcal{A}} \partial_{t} b|^{2}J\nonumber \\
&&~~~~~~~~= \int_{\Omega} (\partial_{t} u \cdot F^{1} + \partial_{t} p  F^{2}+ \partial_{t} b \cdot F^{3})J-\int_{\Sigma} \partial_{t} u \cdot F^{4} + \int_{\Sigma} (\partial_{t}\eta- \sigma \Delta_{\star} \partial_{t}\eta) F^{5}\nonumber\\
&&~~~~~~~~\lesssim (\|\partial_tu\|_{0}\|F^1\|_0+\|\partial_t b\|_0\|F^3\|_0)\|J\|_{L^\infty}+
\|+\|\partial_tu\|_{1/2}\|F^4\|_{-1/2}\\
&&~~~~~~~~~~~~+(\|\partial_t\eta\|_{1/2}+\sigma\|\partial_t\eta\|_{5/2})\|F^5\|_{-1/2}
+\int_\Omega\partial_{t} p  F^{2}J\nonumber\\
&&~~~~~~~~\lesssim\sqrt{\mathcal E}\mathcal D+\int_\Omega\partial_{t} p  F^{2}J.\nonumber
\end{eqnarray*}
Since there is no time derivation on $p$ in $\mathcal D$, for the term involving $\partial_tp$, we have
\begin{equation*}
\begin{aligned}
\int_{\Omega} \partial_{t} p F^{2} J =\frac{d}{dt} \int_{\Omega}pF^{2} J -\int_{\Omega} p\partial_{t} (F^{2} J).
\end{aligned}
\end{equation*}
Then, it follows \eqref{14.2} that
\begin{equation}\label{15.5}
\frac{d}{dt} (\overline{\mathcal{E}}_{(1,0,0)}-\mathcal{F})+ \overline{\mathcal{D}}_{(1,0,0)}\lesssim \sqrt{\mathcal{E}} \mathcal{D},
\end{equation}
where $\overline{\mathcal{E}}_{(1,0,0)}$ and $\overline{\mathcal{D}}_{(1,0,0)}$ are as defined in \eqref{15.1}.

Next, we consider $\alpha\in \mathbb{N}^{1+2}$ with $|\alpha|\leq2$ and $\alpha_{0}=0$, i.e. no temporal derivatives. In this case, we view $(u,b,p,\eta)$ in terms of \eqref{3.13}, which then means that $(v,H,q,h)=(\partial^{\alpha} u,\partial^{\alpha} b,\partial^{\alpha} p,\partial^{\alpha} \eta)$ satisfy \eqref{3.14} with $\Phi^{i}=\partial^{\alpha} G^{i}$ for $i=1,...,5$, where the nonlinearities $G^{i}$ are as defined in \eqref{3.11}-\eqref{3.12}. we may then apply Lemma 3.2 to see that for $|\alpha|\leq 2$ and $\alpha_{0}=0$ we have the identity
\begin{eqnarray}\label{15.6}
&&\frac{d}{dt} \overline{\mathcal{E}_{\alpha}} + \overline{\mathcal{D}_{\alpha}} = \int_{\Omega} (\partial^{\alpha} u \cdot \partial^{\alpha}( G^{1}-\nabla G^2) + \partial^{\alpha} P \partial^{\alpha} G^{2} +\partial^{\alpha} b \cdot \partial^{\alpha} G^{3}) ~\nonumber~\\
&&- \int_{\Sigma} \partial^{\alpha} u \cdot \partial^{\alpha} G^{4}+ \int_{\Sigma}\partial^{\alpha}\eta \partial^{\alpha} G^{5}- \sigma \int_{\Sigma}\partial^{\alpha} G^{5} \Delta_{\star} \partial^{\alpha} \eta.
\end{eqnarray}
When $|\alpha|=2$ and $\alpha_{0}=0$ we write $\partial^{\alpha}= \partial^{\beta+\omega}$ for $|\beta|=|\omega|=1$. We then integrate by parts in the $G^{1}, G^{5}$ terms in \eqref{15.6} to estimate
\begin{eqnarray*}
&&RHS~of~\eqref{15.6}= \int_{\Omega}(-\partial^{\alpha+\beta} u \cdot \partial^{\omega} (G^{1}-\nabla G^2) + \partial^{\alpha} p \partial^{\alpha} G^{2}+\partial^{\alpha+\beta} b \cdot \partial^{\omega} G^{3} ) ~\nonumber~\\
&&~~~~~~~~~~~~~~~~~~~~~-\int_{\Sigma} \partial^{\alpha} u \cdot \partial^{\alpha} G^{4}- \int_\Sigma\partial^{\omega}\eta \partial^{\alpha+ \beta} G^{5}+\sigma \int_\Sigma  \partial^{\alpha+\beta} G^{5} \Delta_{\star} \partial^{\omega} \eta,~\nonumber~\\
&&~~~~~~~~~~~~~~~~~~~\lesssim \|u\|_{3} (\|G^{1}\|_{1}+\|G^{2}\|_{2}) + \|p\|_{2} \|G^{2}\|_{2}+ \|b\|_{3} \|G^{3}\|_{1} + \|D^{2}u\|_{1/2} \|D^{2} G^{4}\|_{\-1/2}~\nonumber~\\
&&~~~~~~~~~~~~~~~~~~~~~ + \|D^{3} G^{5}\|_{-1/2}(\|D \eta\|_{1/2}+\|D^{3}\eta \|_{1/2} )~\nonumber~\\
&&~~~~~~~~~~~~~~~~~~~\lesssim \sqrt{\mathcal{D}}(\|G^{1}\|_{1} +\|G^{2}\|_{2}+ \|G^{3} \|_{1}+ \|G^{4}\|_{3/2}+\|G^{5}\|_{5/2}).
\end{eqnarray*}
The estimate \eqref{14.3} of Lemma \ref{lemm14.2} then tells us that
\begin{equation*}
RHS~of~\eqref{15.6}\lesssim \sqrt{\mathcal{E}} \mathcal{D},
\end{equation*}
and so we have the inequality
\begin{equation}\label{15.7}
\frac{d}{dt} \sum_{|\alpha|=2,\alpha_{0}=0}\overline{\mathcal{E}}_{\alpha}+ \sum_{|\alpha|=2,\alpha_{0}=0}\overline{\mathcal{D}}_{\alpha}\lesssim \sqrt{\mathcal{E}} \mathcal{D}.
\end{equation}

On the other hand, if $|\alpha|<2$ then we must have that $\alpha_{0}=0$, and we can directly apply Lemma \eqref{lemm14.2} to see that
\begin{equation}\label{15.8}
\frac{d}{dt} \sum_{|\alpha|<2,\alpha_{0}=0}\overline{\mathcal{E}}_{\alpha}+ \sum_{|\alpha|<2,\alpha_{0}=0}\overline{\mathcal{D}}_{\alpha}\lesssim \sqrt{\mathcal{E}} \mathcal{D}.
\end{equation}
Now, to deduce \eqref{15.4} we simply sum \eqref{15.5},\eqref{15.6}, and \eqref{15.8}.
\end{proof}
\subsubsection{Enhanced energy estimates}
From the energy-dissipative estimates of Lemma 4.3 we have controlled $\overline{\mathcal{E}}$ and $\overline{\mathcal{D}}$. Our goal now is to show that these can be used to control $\mathcal{E}$ and $\mathcal{D}$ up to some error terms which we will be able to guarantee are small. Here we firstly focus on the estimates for the energies $\mathcal{E}$.

\begin{lemm}\label{lemm4.4}
Let $\mathcal{E}$ be as defined in \eqref{12.1}. Suppose that $\mathcal{E}\leq \delta$, where $\delta\in(0,1)$ is the universal constant given in Lemma \ref{lemm4.1}. Then, we obtain
\begin{equation} \label{15.9}
\mathcal{E} \lesssim \overline{\mathcal{E}}+\mathcal{E}^{2}.
\end{equation}
\end{lemm}
\begin{proof}
According to the definitions of $\overline{\mathcal{E}}$ and $\mathcal{E}$, in order to prove \eqref{15.9} it suffices to prove that
\begin{equation}\label{15.10}
\|u\|_{2}^{2}+ \|p\|_{1}^{2}+\|b\|_{2}^{2}+ \|\partial_{t} \eta\|_{3/2}^{2}+ \|\partial^{2}_{t} \eta\|_{-1/2}^{2}\lesssim \overline{\mathcal{E}}+ \mathcal{E}^{2}.
\end{equation}
For estimating $u$ and $p$ we apply the standard Stokes estimates. Now, according to \eqref{3.13} we have that
\begin{equation}\label{15.21}
\left\{
\begin{aligned}
&-\Delta u+ \nabla p=-\partial_{t} u+ G^{1}~~~~~~~~~~~~~~~~~~~~{\rm in}~\Omega  \\
&{\rm div} v=G^{2}~~~~~~~~~~~~~~~~~~~~~~~~~~~~~~~~~~~~~~~~{\rm in}~\Omega  \\
&(pI-\mathbb{D} u)e_{3}=(\eta I+ \sigma \Delta_{\star} \eta)e_{3} + G^{4}~~~~~~~~{\rm on}~\Sigma\\
&u=0,~~~~~~~~~~~~~~~~~~~~~~~~~~~~~~~~~~~~~~~~~~~~~{\rm on}~\Sigma_{-1},
\end{aligned}
\right.
\end{equation}
and hence we may apply Lemma \ref{lemmA.2} and the estimate \eqref{14.4} of Lemma \ref{lemm4.2} to see that
\begin{equation}
\begin{aligned}
\|u\|_{2}+\|P\|_{1}&\lesssim \|\partial_{t} u\|_{0}+ \|G^{1}\|_{0}+ \|G^{2}\|_{1}+ \|(\eta I+ \sigma \Delta_{\star} \eta)e_{3}\|_{1/2} + \|G^{4}\|_{1/2},\\
&\lesssim \sqrt{\overline{\mathcal{E}}} +\|G^{1}\|_{0}+ \|G^{2}\|_{1}+ \|G^{4}\|_{1/2},\\
&\lesssim \sqrt{\overline{\mathcal{E}}}+ \mathcal{E}.
\end{aligned}
\end{equation}
From this we deduce that the $u,p$ estimates in \eqref{15.10} hold.

Similarly, for estimating $b$, we have
\begin{equation}\label{15.22}
\left\{
\begin{aligned}
&-\Delta b=-\partial_{t}b+ G^{3},~~~~~~~~~~~~~~~~{\rm in}~\Omega \\
&b=0,~~~~~~~~~~~~~~~~~~~~~~~~~~~~~~~~~~{\rm on}~\Sigma \\
&b=0,~~~~~~~~~~~~~~~~~~~~~~~~~~~~~~~~~~{\rm on}~\Sigma_{-1}.
\end{aligned}
\right.
\end{equation}
It follows from Lemma \ref{lemmA.1} that
\begin{eqnarray*}
\|b\|_{2}\lesssim \|\partial_{t} b\|_{0}+ \|G^{3}\|_{0}\lesssim \sqrt{\overline{\mathcal{E}}}+ \mathcal{E}.
\end{eqnarray*}
To estimate the $\partial_{t}\eta$ term in \eqref{15.10} we use the fifth equation of \eqref{3.13} in conjunction with the estimate \eqref{14.4} of Lemma \ref{lemm4.2} and the usual trace estimates to see that
\begin{equation*}
\|\partial_{t} \eta\|_{3/2} \lesssim \|u_{3}\|_{3/2}+ \|G^{5}\|_{3/2}
\lesssim \|u\|_{2}+ \mathcal{E}\lesssim \sqrt{\overline{\mathcal{E}}}+ \mathcal{E}.
\end{equation*}
From this we deduce that the $\partial_{t}\eta$ estimate in \eqref{15.10} holds.

It remains to estimate the $\partial_{t}^{2}\eta$ term in \eqref{15.10}. We apply a temporal derivative to the fifth equation of \eqref{3.13} and integrate against a function $\phi\in H^{\frac{1}{2}}(\Sigma)$ to see that
\begin{equation*}
\int_{\Sigma} \partial^{2}_{t}\eta \phi = \int_{\Sigma} \partial_{t} u_{3}\phi + \int_{\Sigma}\partial_{t} G^{5} \phi .
\end{equation*}
Choose an extension $E\phi\in H^{1}(\Omega)$ with $E\phi|_{\Sigma}=\phi$, $E\phi|_{\Sigma_{-1}}=\phi$, and
$\|E\phi\|_{1}\lesssim \|\phi\|_{1/2}$. Then
\begin{equation*}
\int_{\Sigma} \partial_{t}u_{3} \phi = \int_{\Omega} \partial_{t} u\cdot \nabla_{x}E \phi+ \int_{\Omega}\partial_{t} G^{2} E \phi \leq (\|\partial_{t} u\|_{0}+ \|\partial_{t}G^{2}\|_{-1}) \|\phi\|_{1/2},
\end{equation*}
and Lemma \ref{lemm14.2} implies that
\begin{equation*}
\|\partial^{2}_{t} \eta\|_{-1/2} \lesssim \|\partial_{t}u\|_{0}+ \|\partial_{t} G^{2}\|_{-1}+ \|\partial_{t} G^{5}\|_{-1/2} \lesssim \sqrt{\overline{\mathcal{E}}}+ \mathcal{E}.
\end{equation*}
Then, we complete estimates in \eqref{15.10}.
\end{proof}
\subsubsection{Enhanced dissipate estimates.}
We now show a corresponding result for the dissipation.

\begin{lemm}\label{lemm4.5}
Let $\mathcal{E}$  and $\mathcal{D}$ be as defined in \eqref{12.1} and \eqref{12.2}. Suppose that $\mathcal{E}\leq \delta$, where $\delta\in(0,1)$ is the universal constant given in Lemma \ref{lemm4.1}. Then, we deduce
\begin{equation} \label{15.12}
\mathcal{D} \lesssim \overline{\mathcal{D}}+\mathcal{E}\mathcal{D}.
\end{equation}
\end{lemm}
\begin{proof}
For the dissipation estimates of $u$, we apply the Lemma \ref{A.3} to \eqref{15.21} with $r=3$ and $\phi=-\partial_{t} u+ G^{1}$, $\psi=G^{2}$, $f_{1}= u|_{\Sigma}$, and $f_{2}=0$ and deduce
\begin{equation} \label{15.15}
\|u\|_{3}+ \|\nabla p\|_{1} \lesssim \|-\partial_{t}u+ G^{1}\|_{1}+ \|G^{2}\|_{2}+ \|u\|_{5/2}.
 \end{equation}
 We know that
\begin{equation*}
\|u\|_{1}+ \|Du\|_{1}+ \|D^{2} u\|_{1} \lesssim \sqrt{\overline{\mathcal{D}}},
\end{equation*}
and so trace theory provides us with the estimate
\begin{equation*}
\|u\|_{5/2}\lesssim \sqrt{\overline{\mathcal{D}}}.
\end{equation*}
We also have that $\|\partial_{t} u\|_{1} \lesssim \sqrt{\overline{\mathcal{D}}}$, and Lemma \ref{lemm14.2} tells that
\begin{equation*}
\|G^{1}\|_{1}+ \|G^{2}\|_{2} \lesssim \sqrt{\mathcal{E}\mathcal{D}}.
\end{equation*}
Then, we bring the above estimates into \eqref{15.15} to complete the dissipation estimates of $u$, that is,
\begin{equation}\label{15.23}
  \|u\|_{3}+ \|\nabla p\|_{1} \lesssim \sqrt{\overline{\mathcal{D}}}+ \sqrt{\mathcal{E} \mathcal{D}}.
\end{equation}

For the $b$ dissipative estimate, we directly apply the elliptic estimates to \eqref{15.22} to know
\begin{eqnarray} \label{15.13}
\|b\|_{3} \lesssim \|\partial_{t} b\|_{1}+\|G^{3}\|_{1}\lesssim \sqrt{\overline{\mathcal{D}}} +\|G^{3}\|_{1}\lesssim \sqrt{\overline{\mathcal{D}}}+ \sqrt{\mathcal{E}\mathcal{D}}.
\end{eqnarray}

We now turn to the $\eta$ estimates. For $\alpha \in \mathbb{N}^2$ and $|\alpha|=1$, we apply $\partial^\alpha$ to the fourth equation for \eqref{3.13} to obtain
\begin{equation}
(1-\sigma\Delta_\star)\partial^\alpha\eta=\partial^\alpha p-\partial_3\partial^\alpha u_3-\partial^\alpha G^4_3.
\end{equation}
Then the elliptic estimates, the trace estimates and \eqref{15.23} imply that
\begin{equation}\label{15.24}
\begin{aligned}
\|D\eta\|_{5/2}=\sum _{|\alpha|=1}\|\partial^\alpha\eta\|_{5/2}\lesssim& \sum _{|\alpha|=1}\|\partial^\alpha p-\partial_3\partial^\alpha u_3-\partial^\alpha G^4_3\|\\
\lesssim &\|\nabla p\|_1+\|u\|_3+\|G^4\|_{3/2}
\lesssim \sqrt{\overline{\mathcal{D}}}+ \sqrt{\mathcal{E}\mathcal{D}}.
\end{aligned}
\end{equation}
According to the zero average condition for $\eta$ and by using the Poincar${\rm\acute{e}}$ inequality, we deduce
\begin{equation}\label{15.25}
\|\eta\|_0\leq\|D\eta\|_{0},
\end{equation}
then, \eqref{15.24} and \eqref{15.25} reveal that
\begin{equation} \label{15.16}
\|\eta\|_{7/2} \lesssim \|\eta\|_{0}+ \|D\eta\|_{5/2} \lesssim \|D \eta\|_{5/2} \lesssim \sqrt{\overline{\mathcal{D}}}+ \sqrt{\mathcal{E}\mathcal{D}}.
\end{equation}
For the $\partial_t\eta$ estimates, we use the fifth equation of \eqref{3.13}, \eqref{14.3} and \eqref{5.23} to know
\begin{equation}\label{15.17}
\|\partial_{t} \eta\|_{5/2} \lesssim \|u_{3}\|_{5/2}+ \|G^{5}\|_{5/2} \lesssim \|u\|_{3}+ \|G^{5}\|_{5/2} \lesssim \sqrt{\overline{\mathcal{D}}}+ \sqrt{\mathcal{E}\mathcal{D}},
\end{equation}
and
\begin{equation}\label{15.18}
\|\partial_{t}^{2} \eta\|_{1/2} \lesssim \|\partial_{t}u_{3}\|_{1/2}+ \|\partial_{t}G^{5}\|_{1/2} \lesssim \|\partial_{t}u\|_{1}+ \|\partial_{t}G^{5}\|_{1/2} \lesssim \sqrt{\overline{\mathcal{D}}}+ \sqrt{\mathcal{E}\mathcal{D}}.
\end{equation}

Now we complete the estimate of the pressure by obtaining a bound for $\|p\|_{0}$. To this end we combine the estimates \eqref{15.15} and \eqref{15.16} with the Stokes estimate Lemma \ref{lemmA.2} with $\phi=-\partial_{t} u+ G^{1},\psi=G^{2}$, and $\alpha= (\eta I- \sigma \Delta_{\star} \eta)e_{3} +G^{3} e_{3}$ to bound
\begin{eqnarray*}
&&\|u\|_{3}+ \|P\|_{2}\lesssim \|-\partial_{t} u+ G^{1}\|_{1}+ \|G^{2}\|_{2}+ \|(\eta I-\sigma \delta_{\star}\eta)e_{3}\|_{3/2} ,~\nonumber~\\
&&~~~~~~~~~~~~~~~~\lesssim \|\partial_{t}u\|_{1}+ \|G^{1}\|_{1}+ \|G^{2}\|_{2}+ \|\eta\|_{7/2} ,~\nonumber~\\
&&~~~~~~~~~~~~~~~~\lesssim \sqrt{\overline{\mathcal{D}}}+ \sqrt{\mathcal{E}\mathcal{D}}.
\end{eqnarray*}
Thus,
\begin{equation}\label{15.19}
\|P\|_{2} \lesssim \sqrt{\overline{\mathcal{D}}}+ + \sqrt{\mathcal{E}\mathcal{D}}.
\end{equation}
Finally, \eqref{15.12} follows from \eqref{15.23},\eqref{15.13},\eqref{15.16},\eqref{15.17},\eqref{15.18} and \eqref{15.19}.
\end{proof}

\subsection{Proof of Theorem 2.1}
We now combine the estimates of the previous section in order to deduce our primary a priori estimates for solutions. It shows that under a smallness condition on the energy, the energy decays exponentially and the dissipation integral is bounded by the initial data.
\begin{theo}
Suppose that $(u,b,p,\eta)$ solves \eqref{1.14} on the temporal interval [0,T]. Let $\mathcal{E}$ and $\mathcal{D}$ be as defined in \eqref{12.1} and \eqref{12.2}. Then there exists universal constant $0<\delta_{\star}<\delta$, where $\delta\in(0,1)$ is the universal constant given in Lemma \ref{lemm4.1}, such that if
\begin{equation*}
\sup_{0\leq t\leq T} \mathcal{E}(t) \leq \delta_{\star},
\end{equation*}
then
\begin{equation}\label{16.1}
\sup_{0\leq t\leq T} e^{\lambda t} \mathcal{E}(t)  + \int_{0}^{T} \mathcal{D}(t)dt \lesssim \mathcal{E}(0),
\end{equation}
for all $t\in [0,T]$, where $\lambda>0$ is a universal constant.
\end{theo}
\begin{proof}
According to the definition of $\bar{ \mathcal E}$ and $\bar{\mathcal D}$, Theorem \ref{lemm4.4} and Theorem \ref{lemm4.5}, we find
\begin{equation}\label{15.26}
\bar {\mathcal E}\leq \mathcal E\lesssim \bar{\mathcal E},~~{\rm and} \bar {\mathcal D}\leq \mathcal D\lesssim \bar{\mathcal D},
\end{equation}
as $\delta_\star$ small enough.

Furthermore, substituting \eqref{15.26} into Theorem \ref{lemm4.3}, one has
\begin{equation}
\frac{d}{dt}(\mathcal E-\mathcal F )+\mathcal D\leq 0,
\end{equation}
as $\delta_\star$ small enough.
Moreover, the estimates in \eqref{14.2} tell us that $|\mathcal F|\leq \mathcal E^{3/2}\leq \sqrt\mathcal E\mathcal E$, hence
\begin{equation}\label{15.27}
\frac{d}{dt}\mathcal E+\mathcal D\leq 0.
\end{equation}
On the one hand, we integrate \eqref{15.27} in time over $(0,T)$ to obtain that
\begin{equation}
C\int^T_0\mathcal D(t)dt\leq \mathcal E(T)+C\int^T_0\mathcal D(t)dt\leq \mathcal E(0).
\end{equation}
On the other hand, obviously, we have the bound $\mathcal E\leq\mathcal D$, then we obtain
\begin{equation}
\frac{d}{dt}\mathcal E+\mathcal E\leq 0.
\end{equation}
Then, by using Gronwall's inequality we complete the proof of \eqref{16.1}.

\end{proof}
\section{For the case $\sigma=0$}

\subsection{Nonlinear estimates}

We will employ the form \eqref{3.1} to study the temporal derivative of solutions to \eqref{1.14}. That is, we apply $\partial^{\alpha}$ to \eqref{1.14} to deduce that $(v,~H, ~q,~h)=(\partial^{\alpha}u,\partial^{\alpha}b,\partial^{\alpha}p,\partial^{\alpha}\eta)$ satisfy \eqref{3.1} for certain terms $F^{i}$ for $\partial^{\alpha}=\partial^{\alpha_{0}}_{t}$ with $\alpha_{0}\leq 2N$. Below we record the form of these forcing terms $F^{i}, i=1,2,3,4,5$ for this particular problem, where $F^{1}= \sum_{l=1}^{7} F^{1,l}$, for
\begin{eqnarray}\label{03.1}
&&F_{i}^{1,1}: = \sum_{0<\beta<\alpha} C_{\alpha,\beta} \partial^{\beta}(\partial_{t} \bar{\eta}\tilde{b} K) \partial^{\alpha-\beta} \partial_{3} u_{i} + \sum_{0<\beta\leq\alpha}  C_{\alpha,\beta} \partial^{\alpha-\beta}
\partial_{t} \bar{\eta} \partial^{\beta} (\tilde{b} K) \partial_{3} u_{i}
 ~\nonumber\\
&&F_{i}^{1,2}: = -\sum_{0<\beta\leq\alpha} C_{\alpha,\beta} (\partial^{\beta}(u_{j}\mathcal{A}_{jk}) \partial^{\alpha-\beta} \partial_{k} u_{i} + \partial^{\beta}\mathcal{A}_{ik} \partial^{\alpha-\beta} \partial_{k}p) ~\nonumber~\\
&&F_{i}^{1,3}: = \sum_{0<\beta\leq\alpha} C_{\alpha,\beta}\partial^{\beta} \mathcal{A}_{jl} \partial^{\alpha-\beta}\partial_{l} (\mathcal{A}_{im}\partial_{m} u_{j}+ \mathcal{A}_{jm} \partial_{m} u_{i})  ~\nonumber~\\
&&F_{i}^{1,4}: = \sum_{0<\beta<\alpha} C_{\alpha,\beta}\mathcal{A}_{jk} \partial_{k} (\partial_{\beta} \mathcal{A}_{il} \partial^{\alpha-\beta}\partial_{l} u_{j}+ \partial^{\beta} \mathcal{A}_{jl} \partial^{\alpha-\beta} \partial_{l} u_{i}) ~\\
&&F_{i}^{1,5}: = \partial^{\alpha}\partial_{t} \bar{\eta}\tilde b K \partial_{3} u_{i} ~~{\rm and}~~
F_{i}^{1,6}: =\mathcal{A}_{jk}\partial_{k} (\partial^{\alpha} \mathcal{A}_{il} \partial_{l} u_{j} + \partial^{\alpha} \mathcal{A}_{jl}\partial_{l} u_{i}) ~\nonumber~\\
&&F_{i}^{1,7}: = \sum_{0<\beta\leq\alpha} C_{\alpha,\beta} \partial^{\beta}[(b_j+\bar B_j)\ma_{jk}]\partial^{\alpha-\beta}(\partial_k b_i)\nonumber.
\end{eqnarray}
\begin{equation}\label{3.9}
F^{2,1} := -\sum_{0<\beta<\alpha} C_{\alpha,\beta}\partial^{\beta} \mathcal{A}_{ij} \partial^{\alpha-\beta} \partial_{j} u_{i}, ~~{\rm and}~~F^{2,2}=-\partial^{\alpha}\mathcal{A}_{ij} \partial_{j} u_{i}.
\end{equation}
\begin{eqnarray}\label{03.3}
&&F_i^{3,1}: = \sum_{0<\beta<\alpha} C_{\alpha,\beta} \partial^{\beta}(\partial_{t} \bar{\eta}\tilde{b} K) \partial^{\alpha-\beta} \partial_{3}b_{i} + \sum_{0<\beta\leq\alpha}  C_{\alpha,\beta} \partial^{\alpha-\beta}
\partial_{t} \bar{\eta} \partial^{\beta} (\tilde{b} K) \partial_{3} b_{i}
 ~\nonumber~\\
&&F_i^{3,2}: = -\sum_{0<\beta\leq\alpha} C_{\alpha,\beta} \partial^{\beta}(u_{j}\mathcal{A}_{jk}) \partial^{\alpha-\beta} \partial_{k} b_i  ~\nonumber~\\
&&F_i^{3,3}: = \sum_{0<\beta\leq\alpha} C_{\alpha,\beta}\partial^{\beta} \mathcal{A}_{jl} \partial^{\alpha-\beta} (\mathcal{A}_{jm}\partial_{m} b_i )~\\
&&F_{i}^{3,4}: = -\sum_{0<\beta\leq\alpha} C_{\alpha,\beta} \partial^{\beta}(u_{j}\mathcal{A}_{jk}) \partial^{\alpha-\beta} \partial_{k} b_{i}~\nonumber~\\
&&F_i^{3,5}: = \partial^{\alpha}\partial_{t} \eta\tilde{b} K \partial_{3} b_i ~~{\rm and}~~
F_i^{3,6}: =\mathcal{A}_{jk}\partial^{\alpha} \mathcal{A}_{jl} \partial_{l} b_i,~\nonumber~\\
&&F_{i}^{3,7}: = \sum_{0<\beta\leq\alpha} C_{\alpha,\beta} \partial^{\beta}[(b_j+\bar B_j)\ma_{jk}]\partial^{\alpha-\beta}(\partial_k u_i)\nonumber.
\end{eqnarray}
$F_i^{4}=F_i^{4,1}+F_i^{4,2}$, we have
\begin{eqnarray}\label{03.4}
&&F_{i}^{4,1}: = -(\sum_{0<\beta\leq\alpha} C_{\alpha,\beta}\partial^{\beta}D\eta (\partial^{\alpha-\beta} \eta-\partial^{\alpha-\beta} p),  ~\nonumber~\\
&&F_{i}^{4,2}: = \sum_{0<\beta\leq\alpha} C_{\alpha,\beta}(\partial^{\beta}(\mathcal{N}_{j}\mathcal{A}_{im})\partial^{\alpha-\beta}\partial_{m} u_{j}+
\partial^{\beta}(\mathcal{N}_{j}\mathcal{A}_{jm})\partial^{\alpha-\beta}\partial_{m} u_{i}),
\end{eqnarray}
\begin{equation}\label{03.5}
F^{5} := -\sum_{0<\beta\leq\alpha} C_{\alpha,\beta}\partial^{\beta} D\eta\cdot \partial^{\alpha-\beta} u.
\end{equation}

Now we present the estimates for $F^i~(i=1,\cdots,5)$ when $\partial^\alpha=\partial^{\alpha_0}_t$ for $\alpha_0\leq n.$
\begin{lemm}\label{lemm5.1}
$F^i~(i=1,\cdots,5)$ be defined in \eqref{03.1}-\eqref{03.5}. Let $\partial^{\alpha}= \partial^{\alpha_{0}}_{t}$ with $\alpha_0\leq n$ for $n=2N$ or $n=N+2$. Then, we have
\begin{equation}\label{4.3}
\l F^{1}\r^{2}_{0}+\l F^{2}\r ^{2}_{0}+\l \partial_{t}(J F^{2})\r ^{2}_{0}+\l F^{3}\r ^{2}_{0}+\left\|F^{4}\right\|^{2}_{0}+ \l F^{5}\r ^{2}_{0} \lesssim \mathcal{E}_{2N} \mathcal{D}_{n},
\end{equation}
and
\begin{equation}\label{4.4}
\l F^{2}\r^{2}_{0} \leq \mathcal{E}_{2N}\mathcal{E}_{n}.
\end{equation}
\end{lemm}
\begin{proof}
Firstly, we consider the estimate for $F^1$. Note that each term in the sums is at least quadratic, and each such term can be written in the form $XY$, where $X$ involves fewer derivative counts than $Y$. We may apply the usual Sobolev embeddings Lemmas along with the definitions of $\mathcal{E}_{2N}$ and $\mathcal{D}_n$ to estimate $\l X\r^2_{L^\infty}\lesssim\mathcal E_{2N}$ and $\l Y\r^2_{0}\lesssim\mathcal D_{n}$. Hence $\l XY\r^2_0\leq \l X\r^2_{L^\infty}\l Y\r^2_{0}\lesssim \mathcal E_{2N}\mathcal D_{n}$. The estimates of $F^2$, $F^3$ and \eqref{4.4} are similarly. A similar argument also employing trace estimates obtain the estimates of $F^4$ and $F^5$. The same argument also works for
$\partial_t(JF^{2,1})$. To bound $\partial_t(JF^{2,2})$ for $\alpha_0=n$ we have to estimate $\l\partial^{n+1}_t\ma\r^2_0\lesssim\mathcal D_n$, but this is possible due to Lemma \ref{lemm4.2}. Then a similar splitting into $L^{\infty}$ and $H^0$ estimates shows that $\l\partial_t(JF^{2,2})\r\lesssim \mathcal{E}_{2N}\mathcal{D}_n$, and then we complete the proof of \eqref{4.3}.
\end{proof}

Now, for the case $\sigma=0$, we first estimate the $G^i$ terms defined in \eqref{3.11}-\eqref{3.15} at the $2N$ level.
\begin{lemm}\label{lemm5.2}
Let $G^{1},...,G^{5}$  de defined in \eqref{3.11}-\eqref{3.15}. There exists a $\theta>0$ such that,
\begin{equation}\label{4.7}
\begin{aligned}
&\l\bar{\nabla}^{4N-2}_0G^{1}\r^{2}_{0} +\l\bar{\nabla}^{4N-2}_0G^{2}\r^{2}_{1}+ \l\bar{\nabla}_0^{4N-2}G^{3} \r^{2}_{0} \\
&~~~~~+ \l\bar{D}^{4N-2}_{0} G^{4}\r^{2}_{1/2} \lesssim \mathcal{E}^{1+\theta}_{2N},
\end{aligned}
\end{equation}
\begin{equation}\label{4.8}
\begin{aligned}
&\l\bar{\nabla}^{4N-2}_0G^{1}\r^{2}_{0} +\l\bar{\nabla}_0^{4N-2}G^{2}\r^{2}_{1}+ \l\bar{\nabla}_0^{4N-2}G^{3} \r^{2}_{0} + \l\bar{D}^{4N-2}_{0} G^{4}\r^{2}_{1/2}\\
&+\l\bar{D}^{4N-2}_0 G^{5}\r^{2}_{1/2}
+\l\bar{\nabla}^{4N-3}\partial_{t} G^{1}\r^{2}_{0} +\l\bar{\nabla}^{4N-3} \partial_{t}G^{2}\r^{2}_{1}+ \l\bar{\nabla}^{4N-3}\partial_{t} G^{3} \r^{2}_{0}\\
&~~~~~~~ + \l\bar{D}^{4N-3} \partial_{t} G^{4}\r^{2}_{1/2}+\l\bar{D}^{4N-2} \partial_{t} G^{5}\r^{2}_{1/2}
\lesssim \mathcal{E}^{\theta}_{2N} \mathcal{D}_{2N},
\end{aligned}
\end{equation}
and
\begin{equation}\label{4.9}
\begin{aligned}
&\l{\nabla}^{4N-1}G^{1}\r^{2}_{0} +\l{\nabla}^{4N-1}G^{2}\r^{2}_{1}+ \l{\nabla}^{4N-1}G^{3} \r^{2}_{0}
+ \l{D}^{4N-1} G^{4}\r^{2}_{1/2}\\
&+\l{D}^{4N-1} G^{5}\r^{2}_{1/2} \lesssim \mathcal{E}^{\theta}_{2N}\mathcal D_{2N}
+\mathcal{E}_{N+2}\mathcal F_{2N}.
\end{aligned}
\end{equation}
\end{lemm}

\begin{proof}
These estimates can be proved similar as [8, Theorem 3.3]. 

\end{proof}
Similarly, we can obtain the estimate of $G^i$ terms defined in \eqref{3.11}-\eqref{3.15} at the $N+2$ level as $\sigma=0$.
\begin{lemm}\label{lemm5.3}
Let $G^{1},...,G^{5}$  de defined in \eqref{3.11}-\eqref{3.15}. There exists a $\theta>0$ such that,
\begin{equation}\label{4.10}
\begin{aligned}
&\l\bar{\nabla}^{2(N+2)-2}_0G^{1}\r^{2}_{0} +\l\bar{\nabla}^{2(N+2)-2}_0G^{2}\r^{2}_{1}+ \l\bar{\nabla}_0^{2(N+2)-2}G^{3} \r^{2}_{0} \\
&~~~+ \l\bar{D}^{2(N+2)-2}_{0} G^{4}\r^{2}_{1/2} \lesssim \mathcal{E}^{\theta}_{2N}\mathcal E_{N+2},
\end{aligned}
\end{equation}
and
\begin{equation}\label{4.11}
\begin{aligned}
&\l\bar{\nabla}^{2(N+2)-1}_0G^{1}\r^{2}_{0} +\l\bar{\nabla}_0^{2(N+2)-1}G^{2}\r^{2}_{1}+ \l\bar{\nabla}_0^{2(N+2)-1}G^{3} \r^{2}_{0} + \l\bar{D}^{{2(N+2)-1}} G^{4}\r^{2}_{1/2}\\
&~~~~~~~+\l\bar{D}^{2(N+2)-1}_0 G^{5}\r^{2}_{1/2}
+\l\bar{D}^{2(N+2)-2} \partial_{t} G^{5}\r^{2}_{1/2}
\lesssim \mathcal{E}^{\theta}_{2N} \mathcal{D}_{N+2},
\end{aligned}
\end{equation}
\end{lemm}

\subsection{Energy evolution}

We define the temporal energy and dissipation, respectively, as
\begin{eqnarray}\label{5.1}
\bar{\mathcal{E}}_{n}^{0} := \sum_{j=0}^{n} (\l\sqrt{J} \partial_{t}^{j} u \r_{0}^{2} +\l\sqrt{J} \partial_{t}^{j} b\r_{0}^{2}+\l \partial_{t}^{j} \eta \r_{0}^{2} )~
\end{eqnarray}
\begin{eqnarray}\label{5.2}
\bar{\mathcal{D}}_{n}^{0} := \sum_{j=0}^{n} (\l \mathbb{D}\partial_{t}^{j} u \r_{0}^{2} +\l \nabla \partial_{t}^{j} b \r_{0}^{2}).
\end{eqnarray}
Then, we define the horizontal energies and dissipation, respectively, as
\begin{equation}\label{5.3}
\begin{aligned}
\bar{\mathcal{E}}_{n}:=& \left\| \bar{D}_{0}^{2n-1}u\right\| _{0}^{2} +\l D \bar{D}^{2n-1}u\r _{0}^{2}+\l \bar{D}_{0}^{2n-1}b\r _{0}^{2} \\
&+\l D\bar{D}^{2n-1}b\r _{0}^{2}+\l \bar{D}^{2n-1}\eta\r _{0}^{2} +\l D \bar{D}^{2n-1}\eta\r _{0}^{2},
\end{aligned}
\end{equation}
and
\begin{equation}\label{5.4}
\begin{aligned}
\bar{\mathcal{D}}_{n}:=& \l \bar{D}_{0}^{2n-1} \mathbb{D}(u)\r_{0}^{2} +\l D \bar{D}
^{2n-1}\mathbb{D}(u)\r_{0}^{2}+\l \bar{D}_{0}^{2n-1} \nabla b\r_{0}^{2} +\l D\bar{D}
^{2n-1}\nabla b\r_{0}^{2}.
\end{aligned}
\end{equation}

\subsubsection{Energy Evolution of Temporal Derivatives}

First, we present the temporal derivatives estimates at $2N$ level.
\begin{lemm}\label{lemm5.4}
There exist a $\theta>0$ so that
\begin{equation}\label{5.5}
\bar{\mathcal{E}}_{2N}^{0} (t) + \int_{0}^{t} \bar{\mathcal{D}}_{2N}^{0} \lesssim \mathcal{E}_{2N}(0)+
(\mathcal{E}_{2N}(t))^{3/2} + \int_{0}^{t} (\mathcal{E}_{2N})^{\theta}\mathcal{D}_{2N}.
\end{equation}
\end{lemm}
\begin{proof}
We apply $\partial^{\alpha}=\partial^{\alpha_0}_t$ with $0\leq \alpha\leq 2N$ to \eqref{1.14} and set $v=\partial_t^{\alpha_0}u$, $q=\partial_t^{\alpha_0}p$, $H=\partial_t^{\alpha_0}b$, $h=\partial_t^{\alpha_0}\eta$ satisfying \eqref{3.1}. Then, according to Lemma \ref{lemm3.1} and integrating in time from 0 to $t$, we deduce
\begin{align*}
 \int_{\Omega}& \left(\frac{|\partial^{\alpha_{0}}_{t} u|^{2}}{2}+ \frac{|\partial_{t}^{\alpha_{0}} b|^{2}}{2}\right)J + \int_{\Sigma} \frac{|\partial^{\alpha_{0}}_{t}\eta|^{2}}{2}+ \int^t_0\int_{\Omega} \left(\frac{|\mathbb{D}_{\mathcal{A}} \partial^{\alpha_{0}}_{t} u|^{2}}{2}+|\nabla_{\mathcal{A}} \partial^{\alpha_{0}}_{t} b|^{2}\right)J\\
=&\int_{\Omega} \left(\frac{|\partial^{\alpha_{0}}_{t} u(0)|^{2}}{2}+\frac{|\partial_{t}^{\alpha_{0}} b(0)|^{2}}{2}\right)J + \int_{\Sigma} \frac{|\partial^{\alpha_{0}}_{t}\eta(0)|^{2}}{2}+\int^t_0 \int_{\Sigma} (-\partial^{\alpha_{0}}_{t} u \cdot F^{4}+ \partial^{\alpha_{0}}_{t}\eta F^{5})\\
&~~~~~+ \int^t_0\int_{\Omega} (\partial^{\alpha_{0}}_{t} u \cdot F^{1} + \partial^{\alpha_{0}}_{t} p F^{2}+ \partial^{\alpha_{0}}_{t}b\cdot F^{3})J.
\end{align*}
Next, we will estimate the right hand side terms involving $F^i$ of the above equation. For the $F^1$ term, according to Lemma \ref{lemm4.1} and Lemma \ref{lemm5.1}, one has
\begin{align*}
\int^t_0\int_\Omega\partial^{\alpha_{0}}_{t} u\cdot F^1 J\lesssim&\int^t_0\l\partial^{\alpha_{0}}_{t} u\r_0\|J\|_{L^\infty}\|F^1\|_0\\
\lesssim&\int^t_0\sqrt{\mathcal{D}_{2N}}\sqrt{\mathcal{E}_{2N}\mathcal{D}_{2N}}=\int^t_0\sqrt{\mathcal{E}_{2N}}\mathcal{D}_{2N}.
\end{align*}
Similarly,
\begin{equation}
\int^t_0\int\Omega\partial^{\alpha_{0}}_{t} u\cdot F^3 \lesssim\int^t_0\sqrt{\mathcal{E}_{2N}}\mathcal{D}_{2N},
\end{equation}
and
\begin{equation}
\begin{aligned}
 \int_{\Sigma} (-\partial^{\alpha_{0}}_{t} u \cdot F^{4}+ \partial^{\alpha_{0}}_{t}\eta F^{5})
 \lesssim &\int_0^t(\|\partial^{\alpha_0}_t u\|_{H^0(\Sigma)}\|F^4\|_0+\|\partial^{\alpha_0}_t\eta\|_0\|F^5\|_0)\\
\lesssim&\int^t_0\sqrt{\mathcal{D}_{2N}}\sqrt{\mathcal{E}_{2N}\mathcal{D}_{2N}}
 =\int^t_0\sqrt{\mathcal{E}_{2N}}\mathcal{D}_{2N}.
 \end{aligned}
\end{equation}
For the term $\partial^{\alpha_0}_tp F^2$, when $\alpha_0=2N$, there is one more time derivative on $p$ than can be controlled by $\mathcal{D}_{2N}$. Hence, we have to consider the cases $\alpha_0<2N$ and $\alpha_0=2N$ separately. In the
case $\alpha_0=2N$, we have
\begin{equation}
\begin{aligned}
\int^t_0\int_\Omega\partial^{2N}_tp F^2
=&-\int^t_0\int_\Omega\partial^{2N-1}_tp\partial_t(JF^2)+\int_\Omega(\partial^{2N-1}_tpJF^2)(t)\\
&-\int_\Omega(\partial^{2N-1}_tpJF^2)(0).
\end{aligned}
\end{equation}
According to Lemma \ref{lemm5.1}, one has
\begin{align}\label{5.9}
-\int^t_0\int_\Omega\partial^{2N-1}_tp\partial_t(JF^2)\lesssim &\int^t_0\l\partial^{2N-1}_tp\r_0\l\partial_t(JF^2)\r_0\\
\lesssim&\int^t_0\sqrt{\mathcal{D}_{2N}}\sqrt{\mathcal{E}_{2N}\mathcal{D}_{2N}}
 =\int^t_0\sqrt{\mathcal{E}_{2N}}\mathcal{D}_{2N}.\nonumber
 \end{align}
Then, it follows from \eqref{4.4} and Lemma \ref{lemm4.1} that
\begin{align}\label{5.10}
\int_\Omega(\partial^{2N-1}_tpJF^2)(t)\lesssim \l\partial^{2N}_tp\r_0\l F^2\r_0\|J\|_{L^\infty}\lesssim (\mathcal{E}_{2N})^{3/2}.
\end{align}
Combining the estimates \eqref{5.9} and \eqref{5.10}, we obtain
\begin{equation}\label{5.11}
\int^t_0\int_\Omega\partial^{2N}_tp F^2J\lesssim \mathcal{E}_{2N}(0)+(\mathcal{E}_{2N})^{3/2}+\int^t_0\sqrt{\mathcal{E}_{2N}}\mathcal{D}_{2N}.
\end{equation}

 In the other case $0\leq\alpha_0<2N$, by using \eqref{4.3}, we directly have
 \begin{equation}\label{5.11}
\int^t_0\int_\Omega\partial^{\alpha_0}_tp F^2J\lesssim
\int^t_0\l\partial^{\alpha_0}_tp\r\|F^2\|_0
\lesssim\int^t_0\sqrt{\mathcal{D}_{2N}}\sqrt{\mathcal{E}_{2N}\mathcal{D}_{2N}}
 =\int^t_0\sqrt{\mathcal{E}_{2N}}\mathcal{D}_{2N}.
 \end{equation}
 Furthermore, according to Lemma \ref{lemm4.1} we can easily deduce
 \begin{align}
 \int^t_0\int_\Omega\frac{|\mathbb{D}\partial^{\alpha_0}_tu|^2}{2}J
 \lesssim\int^t_0\int_\Omega\frac{|\mathbb{D}_\ma\partial^{\alpha_0}_tu|^2}{2}J+\int^t_0\sqrt\mathcal{E}_{2N}\mathcal{D}_{2N},
 \end{align}
 and
 \begin{align}
 \int^t_0\int_\Omega{|\nabla\partial^{\alpha_0}_tb|^2}J
 \lesssim\int^t_0\int_\Omega{|\nabla_\ma\partial^{\alpha_0}_tb|^2}J+\int^t_0\sqrt\mathcal{E}_{2N}\mathcal{D}_{2N}.
 \end{align}
Therefore, we complete the proof of Lemma \ref{lemm5.4}.

\end{proof}

Now, we present the corresponding estimates at the $N+2$ level.
\begin{lemm}\label{lemm5.5}
In the case $0\leq \alpha_{0} \leq N+2$, we have
\begin{equation}
\partial_t (\bar{\mathcal{E}}^{0}_{N+2}-2\int_\Omega\partial_{t}^{N+1}pF^2J) + \bar{\mathcal{D}}^{0}_{N+2}\lesssim \sqrt{\mathcal{E}_{2N}} \mathcal{D}_{N+2}.
\end{equation}

\end{lemm}
\begin{proof}
The proof of Lemma \ref{lemm5.5} is similar to Lemma \ref{lemm5.4}. Here, for brevity, we omit the proof.
\end{proof}

\subsubsection{Energy Evolution of Horizontal derivatives}

In this subsection, we will show how the horizontal energies evolve at the $2N$ and $N+2$ level, respectively.
\begin{lemm}\label{lemm5.6}
Let $\alpha\in \mathbb{N}^{1+2}$, $0\leq \alpha_0\leq 2N-1$ and $|\alpha|\leq 4N$. Then, there exist a $\theta>0$ so that
\begin{equation}\label{5.16}
\bar{\mathcal{E}}_{2N} (t) + \int_{0}^{t} \bar{\mathcal{D}}_{2N} \lesssim{ \mathcal{E}}_{2N}(0)+
\int^t_0(\mathcal{E}_{2N})^{\theta}\mathcal{D}_{2N} + \int_{0}^{t}\sqrt{ \mathcal{D}_{2N}\mathcal{F}_{2N}\mathcal{E}_{N+2}}.
\end{equation}
\end{lemm}
\begin{proof}
We apply $\partial^{\alpha}(\alpha\in\mathbb{N}^{1+2}$,$0\leq \alpha_0\leq 2N-1$ and $|\alpha|\leq 4N$) to \eqref{3.13} and set $v=\partial_t^{\alpha_0}u$, $q=\partial_t^{\alpha_0}p$, $H=\partial_t^{\alpha_0}b$, $h=\partial_t^{\alpha_0}\eta$ satisfying \eqref{3.14} with $\Phi^i=\partial^\alpha G_i~(i=1,\cdots 5)$. Then, according to Lemma \ref{lemm3.2} and integrating in time from 0 to $t$, we obtain
\begin{equation}\label{5.17}
\begin{aligned}
&\partial_t \int_{\Omega} \left(\frac{|\partial^{\alpha} u|^{2}}{2} + \frac{|\partial^{\alpha}\eta|^{2}}{2}+ \frac{|\partial^{\alpha}b|^{2}}{2}\right)+ \int_{\Omega}\left( \frac{|\mathbb{D} \partial^{\alpha} u|^{2}}{2}+  |\nabla \partial^{\alpha} b|^{2}\right) \\
&~~= \int_{\Omega} \partial^{\alpha} u \cdot (\partial^{\alpha} G^{1}-\nabla\partial^\alpha G^2) +\int_{\Omega} \partial^{\alpha} p \partial^{\alpha} G^{2} +\int_{\Omega} \partial^{\alpha}b\cdot \partial^{\alpha} G^{3}\\
&~~~~~~~~~+ \int_{\Sigma} (-\partial^{\alpha} u \cdot \partial^{\alpha} G^{4}+ \partial^{\alpha} \eta \cdot \partial^{\alpha} G^{5}).
\end{aligned}
\end{equation}

We first consider the case $0\leq \alpha_0\leq 2N-1$, $|\alpha|\leq 4N-1$. According to \eqref{4.7}, one has
\begin{align}\label{5.18}
&\int_{\Omega} \partial^{\alpha} u \cdot (\partial^{\alpha} G^{1}-\nabla\partial^\alpha G^2) +\int_{\Omega} \partial^{\alpha} p \partial^{\alpha} G^{2} +\int_{\Omega} \partial^{\alpha}b\cdot \partial^{\alpha} G^{3}\nonumber\\
\lesssim~& \|\partial^\alpha u\|_0(\|\partial^\alpha G^1\|_0+\|\partial^\alpha G^2\|_1)+\|\partial^\alpha b\|_0\|\partial^\alpha G^3\|_0+\|\partial^\alpha p\|_0\|\partial^\alpha G^2\|_0\\
\lesssim~&\sqrt{\mathcal{D}_{2N}}\sqrt{\mathcal{E}^\theta_{2N}\mathcal{D}_{2N}+\mathcal{E}_{N+2}\mathcal{F}_{2N}}
\lesssim~\mathcal{E}^{\theta/2}_{2N}\mathcal{D}_{2N}
+\sqrt{\mathcal{D}_{2N}\mathcal{F}_{2N}\mathcal{E}_{N+2}}.\nonumber
\end{align}

It follows from the trace estimate $\|\partial^\alpha u\|_{H^0(\Sigma)}\lesssim \|\partial^\alpha u\|_1\lesssim\sqrt{\mathcal{D}_{2N}}$ that
\begin{align}\label{5.19}
\int_\Sigma(-\partial^{\alpha} u& \cdot \partial^{\alpha} G^{4}+ \partial^{\alpha} \eta \cdot \partial^{\alpha} G^{5})
\lesssim~\|\partial^\alpha u\|_{H^0(\Sigma)}\|\partial^\alpha G^4\|_0+\|\partial^\alpha \eta\|_0\|\partial^\alpha G^5\|_0
\nonumber\\
\lesssim~&\sqrt{\mathcal{D}_{2N}}\sqrt{\mathcal{E}^\theta_{2N}\mathcal{D}_{2N}+\mathcal{E}_{N+2}\mathcal{F}_{2N}}
\lesssim~\mathcal{E}^{\theta/2}_{2N}\mathcal{D}_{2N}
+\sqrt{\mathcal{D}_{2N}\mathcal{F}_{2N}\mathcal{E}_{N+2}}.
\end{align}

For the other case $|\alpha|=4N$, since $\alpha_0\leq 2N-1$, we can write $\alpha=\beta+(\alpha-\beta)$, where
$\eta\in \mathbb{N}^2$ satisfies $|\beta|=1$. Hence, $\partial^\alpha$ involves at least one spatial derivative.
Since $|\alpha-\beta|=4N-1$, we can integrate by parts and using \eqref{4.7} to find that
\begin{align}\label{5.20}
&\left|\int_{\Omega} \partial^{\alpha} u \cdot (\partial^{\alpha} G^{1}-\nabla\partial^\alpha G^2)+\int_{\Omega} \partial^{\alpha}b\cdot \partial^{\alpha} G^{3}\right|+\left|\int_{\Omega} \partial^{\alpha} p \partial^{\alpha} G^{2}\right|\nonumber\\
\lesssim~& \left|\int_\Omega\partial^{\alpha+\beta} u\cdot\partial^{\alpha-\beta} G^1
+\partial^{\alpha+\beta} b\cdot\partial^{\alpha-\beta} G^3-\partial^{\alpha+\beta} u\cdot\nabla\partial^{\alpha-\beta} G^2\right|+\left|\int_\Omega\partial^\alpha p\partial^{\alpha-\beta +\beta}G^2\right|\\
\lesssim~& \l\partial^{\alpha+\beta} u\r_0\left(\l\partial^{\alpha-\beta} G^1\r_0+\l\partial^{\alpha-\beta} G^2\r_1\right)+\l\partial^{\alpha+\beta} b\r_0\l\partial^{\alpha-\beta} G^3\r_0
+\l\partial^\alpha p\r_0\l\bar\nabla ^{4N-1}G^2\r_1\nonumber\\
\lesssim~&\mathcal{E}^{\theta/2}_{2N}\mathcal{D}_{2N}
+\sqrt{\mathcal{D}_{2N}\mathcal{F}_{2N}\mathcal{E}_{N+2}}.\nonumber
\end{align}
Integrating by parts and using the trace theorem to find that
\begin{equation}\label{5.21}
\begin{aligned}
\int_{\Sigma} \partial^{\alpha}u \partial^{\alpha} G^{4} =&\left|\int_{\Sigma} \partial^{\alpha+\beta}u \partial^{\alpha-\beta} G^{4}\right|\lesssim \l\partial^{\alpha+\beta}u \r_{H^{-1/2}}\l \partial^{\alpha-\beta} G^{4}\r_{1/2}\\
\lesssim& \|\partial^{\alpha} u \|_{H^{1/2}(\Sigma)}\| \bar{D}^{4N-1} G^{4}\|_{1/2}
\lesssim\|\partial^\alpha u\|_1\| \bar{D}^{4N-1} G^{4}\|_{1/2}\\
\lesssim&\mathcal{E}^{\theta/2}_{2N}\mathcal{D}_{2N}
+\sqrt{\mathcal{D}_{2N}\mathcal{F}_{2N}\mathcal{E}_{N+2}}.
\end{aligned}
\end{equation}

For the term involves $\eta$, we need to apart it into two cases $\alpha_0\geq 1$ and $\alpha_0=0$. In the former case, there is at least one temporal derivative in $\partial^\alpha$. Thus, we have
\begin{equation*}
\|\partial^\alpha\eta\|_{1/2}\lesssim\|\partial^{\alpha-2}\partial_t\eta\|_{1/2}\lesssim\sqrt{\mathcal{D}_{2N}}.
\end{equation*}
Then, integrating by parts, one has
\begin{equation}\label{5.22}
\begin{aligned}
\int_{\Sigma} \partial^{\alpha}\eta \partial^{\alpha} G^{5} \lesssim&\left|\int_{\Sigma} \partial^{\alpha+\beta} \eta \partial^{\alpha-\beta} G^{5}\right|\lesssim \l\partial^{\alpha+\beta} \eta \r_{-1/2}\l \partial^{\alpha-\beta} G^{5}\r_{1/2}\\
\lesssim& \|\partial^{\alpha} \eta \|_{1/2}\| \partial^{\alpha-\beta} G^{5}\|_{1/2}\lesssim~\mathcal{E}^{\theta/2}_{2N}\mathcal{D}_{2N}
+\sqrt{\mathcal{D}_{2N}\mathcal{F}_{2N}\mathcal{E}_{N+2}}.
\end{aligned}
\end{equation}

In the other case $\alpha_0=0$, $\partial^\alpha$ involves only spatial derivatives, We need to analyze in detail. Firstly, we denote
\begin{align*}
-\partial^\alpha G^5=\partial^\alpha(D\eta\cdot u)=&D\partial^\alpha\cdot u+\sum_{0<\beta\leq\alpha,|\beta|=1}C_{\alpha,\beta}D\partial^{\alpha-\beta}\eta\cdot\partial^\beta u\\
&+\sum_{0<\beta\leq\alpha,|\beta|\geq2}C_{\alpha,\beta}D\partial^{\alpha-\beta}\eta\cdot\partial^\beta u\\
=&I_1+I_2+I_3.
\end{align*}
By integrating by parts, one has
\begin{align*}
\left|\int_\Sigma \partial ^\alpha\eta I_1\right|=&\frac{1}{2}\left|\int_\Sigma D|\partial^\alpha\eta|^2\cdot u\right|
=\frac{1}{2}\left|\int_\Sigma \partial^\alpha\eta \partial^\alpha\eta(\partial_1u_1+\partial_2u_2)\right|\\
\lesssim&\l\partial^\alpha\eta\r_{1/2}\l\partial^\alpha\eta\r_{-1/2}\|\partial_1u_1+\partial_2u_2\|_{L^\infty}\\
\lesssim&\|\eta\|_{4N+1/2}\|D\eta\|_{4N-3/2}\mathcal{E}_{N+2}\\
\lesssim&\sqrt{\mathcal{D}_{2N}\mathcal{F}_{2N}\mathcal{E}_{N+2}}.
\end{align*}
Similarly, for the estimate of $I_2$, we have
\begin{equation*}
\left|\int_\Sigma \partial ^\alpha\eta I_2\right|\lesssim\sqrt{\mathcal{D}_{2N}\mathcal{F}_{2N}\mathcal{E}_{N+2}}.
\end{equation*}
Finally, for $I_3$, we find that
\begin{align*}
\left|\int_\Sigma \partial ^\alpha\eta I_3\right|\lesssim \|\partial^\alpha\eta\|_{-1/2}\|D\partial^{\alpha-\beta\eta}\cdot\partial^\beta u\|_{H^{1/2}(\Sigma)}\lesssim\sqrt{\mathcal{D}_{2N}}\sqrt{\mathcal{E}_{2N}
\mathcal{D}_{2N}}=\sqrt{\mathcal{E}_{2N}}\mathcal{D}_{2N}.
\end{align*}
Thus, we deduce
\begin{equation}\label{5.23}
\left|\int_\Sigma \partial ^\alpha\eta \partial^\alpha G^5\right|\lesssim
\sqrt{\mathcal{E}_{2N}}\mathcal{D}_{2N}+\sqrt{\mathcal{D}_{2N}\mathcal{F}_{2N}\mathcal{E}_{N+2}}.
\end{equation}
Bring the estimates \eqref{5.18}-\eqref{5.23} into \eqref{5.17}, we conclude
\begin{equation}\label{5.24}
\begin{aligned}
\partial_t \int_{\Omega}& \left(\frac{|\partial^{\alpha} u|^{2}}{2} + \frac{|\partial^{\alpha}\eta|^{2}}{2}+ \frac{|\partial^{\alpha}b|^{2}}{2}\right)+ \int_{\Omega}\left( \frac{|\mathbb{D} \partial^{\alpha} u|^{2}}{2}+  |\nabla \partial^{\alpha} b|^{2}\right)\\
 &~~~~~~~~~\lesssim (\mathcal{E}_{2N})^{\theta/2}\mathcal{D}_{2N}+\sqrt{\mathcal{D}_{2N}\mathcal{F}_{2N}\mathcal{E}_{N+2}},
\end{aligned}
\end{equation}
and then \eqref{5.16} follows from \eqref{5.24}.
\end{proof}

Similar to the estimates in Lemma \ref{lemm5.6}, by using \eqref{4.9}, we can obtain the horizontal energies estimates
corresponding estimate at the $N+2$ level, namely,
\begin{lemm}\label{lemm5.7}
Let $\alpha\in\mathbb{N}^{1+2}$ satisfy $\alpha_0\leq N+1$ and $|\alpha|\leq 2(N+2)$. Then,
\begin{equation}
\partial_t(\l\partial^\alpha u\r^2_0+\l\partial^\alpha b\r^2_0+\l\partial^\alpha \eta\r^2_0)+\|\mathbb{D}\partial^\alpha u\|^2_0+\|\nabla\partial^\alpha b\|^2_0\lesssim (\mathcal{E}_{2N})^{\theta/2}\mathcal{D}_{N+2}.
\end{equation}
Furthermore, we deduce
\begin{equation}
\partial_t\bar{\mathcal{E}}_{N+2}+\bar{\mathcal{D}}_{N+2}\lesssim (\mathcal{E}_{2N})^{\theta/2}\mathcal{D}_{N+2}.
\end{equation}
\end{lemm}

\subsubsection{Energy improvement}
In this subsection, we will show that, up to some error terms, the total energy $\mathcal{E}_{n}$ can be controlled by $\bar{\mathcal{E}}_n+\bar{\mathcal{E}}^0_n$ and the total dissipation $\mathcal{D}_{2N}$ can be bounded by $\bar{\mathcal{D}}_n+\bar{\mathcal{D}}^0_n$, respectively.

\begin{lemm}\label{lemm5.8}
There exists a $\theta>0$ so that
\begin{equation}\label{5.27}
\mathcal{E}_{2N}\lesssim\bar{\mathcal{E}}_{2N}+\bar{\mathcal{E}}^0_{2N}+(\mathcal{E}_{2N})^{\theta+1},
\end{equation}
and
\begin{equation}\label{5.28}
\mathcal{E}_{N+2}\lesssim\bar{\mathcal{E}}_{N+2}+\bar{\mathcal{E}}^0_{N+2}+(\mathcal{E}_{2N})^{\theta}\mathcal E_{N+2}.
\end{equation}
\end{lemm}

\begin{proof}
We let $n$ denote either $2N$ or $N+2$ throughout the proof, and we define
\begin{equation*}
W_n=\sum^{n-1}_{j=0}\left(\l\partial^j_tG^1\r^2_{2n-2j-2}+\l\partial^j_tG^2\r^2_{2n-2j-1}
+\l\partial^j_tG^3\r^2_{2n-2j-2}+\l\partial^j_tG^4\r^2_{2n-2j-3/2}\right)
\end{equation*}
According to the definitions of $\bar{\mathcal{E}}^0_n$ and $\bar{\mathcal{E}}_n$, we know
\begin{equation}\label{5.29}
\l\partial^n_t u\r^2_0+\l\partial^n_t b\r^2_0+\sum^n_{j=0}\l\partial^j_t\eta\r^2_{2n-2j}\lesssim \bar{\mathcal E}^0_n+\bar{\mathcal{E}}_n.
\end{equation}
To control $u$ and $P$, we will apply the standard Stokes estimates. According to \eqref{3.13} we have the form
\begin{equation}\label{5.30}
\left\{
\begin{aligned}
&-\Delta u+ \nabla p=-\partial_{t} u+ G^{1},~~~~~~~~~~~~~~~~~~~~~~~~~~{\rm in}~\Omega~ \\
&{\rm div} u=G^{2},~~~~~~~~~~~~~~~~~~~~~~~~~~~~~~~~~~~~~~~~~~~~~~{\rm in}~\Omega~ \\
&(pI-\mathbb{D} u)e_{3}=\eta e_3+G^{4},~~~~~~~~~~~~~~~~~~~~~~~~~~~~{\rm on}~\Sigma~\\
&u=0 ~~~~~~~~~~~~~~~~~~~~~~~~~~~~~~~~~~~~~~~~~~~~~~~~~~~~~ {\rm on}~\Sigma_{-1}.
\end{aligned}
\right.
\end{equation}
Then we apply $\partial^j_t~(j=0,1,\cdots,n-1)$ to \eqref{5.30} and we use Lemma \ref{lemmA.2} to find that
\begin{equation} \label{5.31}
\begin{aligned}
\l\partial^{j}_{t} u\r_{2n-2j}^{2}+\l\partial^{j}_{t} p\r_{2n-2j-1}^{2} &\lesssim \l\partial_{t}^{j+1} u\r_{2n-2j-2}^{2}+ \l\partial^{j}_{t}G^{1}\r_{2n-2j-2}^{2}+ \l\partial^{j}_{t}G^{2}\r_{2n-2j-1}^{2}\\
&~~~~+ \l\partial^{j}_{t}\eta\r_{2n-2j-3/2}^{2} + \l\partial^{j}_{t}G^{4}\r_{2n-2j-3/2}^{2}\\
&\lesssim \l \partial_{t}^{j+1} u \r^{2}_{2n-2(j+1)}+\bar{\mathcal{E}}_n+\bar{\mathcal{E}}^0_n+ \mathcal{W}_{n}.
\end{aligned}
\end{equation}
To control $b$, we will apply the standard elliptic estimate, and according to \eqref{3.13} $b$ satisfies
\begin{equation}\label{5.32}
\left\{
\begin{aligned}
&-\Delta b=\partial_t b+G^3, ~~~~~~~~~~~~~~~~~~~~~~~~~~{\rm in}~\Omega~ \\
&b=0,~~~~~~~~~~~~~~~~~~~~~~~~~~~~~~~~~~~~~~~~~~{\rm on}~\Sigma\cup\Sigma_{-1}.~ \\
\end{aligned}
\right.
\end{equation}
Then, applying $\partial^j_t$ to \eqref{5.32} and using Lemma \ref{lemmA.1}, one has
\begin{equation} \label{5.33}
\begin{aligned}
\l\partial^{j}_{t} b\r_{2n-2j}^{2} &\lesssim \l\partial_{t}^{j+1}b\r_{2n-2j-2}^{2}+ \l\partial^{j}_{t}G^{3}\r_{2n-2j-2}^{2}\\
&\lesssim \l \partial_{t}^{j+1} b \r^{2}_{2n-2(j+1)}+ \mathcal{W}_{n}.
\end{aligned}
\end{equation}
Combining \eqref{5.31} and \eqref{5.33}, we use the estimates obtained in \eqref{4.7} and \eqref{4.10} to obtain
\begin{equation} \label{5.34}
\begin{aligned}
\l\partial^{j}_{t} u\r_{2n-2j}^{2}&+\l\partial^{j}_{t} p\r_{2n-2j-1}^{2}+\l\partial^{j}_{t} b\r_{2n-2j}^{2}\\
\lesssim &\l\partial_{t}^{j+1} u\r_{2n-2(j+1)}^{2}+ \l \partial_{t}^{j+1} b \r^{2}_{2n-2(j+1)}+\bar{\mathcal{E}}_n+\bar{\mathcal{E}}^0_n+ \mathcal{W}_{n}.
\end{aligned}
\end{equation}
After a simple induction on \eqref{5.34}, we yield that
\begin{equation}\label{5.35}
\begin{aligned}
\sum^{n-1}_{j=0}&\left(\l\partial^j_tu\r^2_{2n-2j}+\l\partial^{j}_{t} p\r_{2n-2j}^{2}+\l\partial^{j}_{t} b\r_{2n-2j}^{2}\right)\\
&\lesssim \bar{\mathcal{E}}_n+\bar{\mathcal{E}}^0_n+ \mathcal{W}_{n}+\l\partial^n_t u\r^2_0+\l\partial^n_tb\r^2_0\\
&\lesssim\bar{\mathcal{E}}_n+\bar{\mathcal{E}}^0_n+ \mathcal{W}_{n}.
\end{aligned}
\end{equation}
Thus, it follows from \eqref{5.29} and \eqref{5.35} that
\begin{equation*}
\mathcal{E}_n\lesssim\bar{\mathcal{E}}_n+\bar{\mathcal{E}}^0_n+ \mathcal{W}_{n}.
\end{equation*}
Finally, for $n=2N$, we employ \eqref{4.7} to bound $\mathcal W_{2N}\lesssim \mathcal E^{1+\theta}_{2N}$. Thus, the estimate and \eqref{5.35} imply \eqref{5.27}. Similarly, for $n=N+2$, we employ \eqref{4.10} to bound $\mathcal W_{N+2}\lesssim \mathcal E^{\theta}_{2N}\mathcal E_{N+2}$. Hence, the estimate and \eqref{5.35} imply \eqref{5.28}.

\end{proof}

\subsubsection{Dissipation improvement}
\begin{lemm}\label{lemm5.9}
There exists a $\theta>0$, so that
\begin{equation}\label{5.36}
\mathcal{D}_{2N}\lesssim \bar{\mathcal{D}}_{2N}+\bar{\mathcal{D}}_{2N}^{0}+\mathcal{F}_{2N}\mathcal{E}_{N+2}+\mathcal{E}_{2N}^{\theta}\mathcal{D}_{2N},
\end{equation}
and
\begin{equation}\label{5.37}
\mathcal{D}_{N+2}\lesssim \bar{\mathcal{D}}_{N+2}+\bar{\mathcal{D}}_{N+2}^{0}+\mathcal{E}_{2N}^{\theta}\mathcal{D}_{N+2}.
\end{equation}
\end{lemm}
\begin{proof}
Let $n$ denote $2N$ or $N+2$, and define
\begin{equation*}
\begin{aligned}
\mathcal{Y}_n=&\l\bar{\nabla}^{2n-1}_0G^1\r^{2}_0+\l\bar{\nabla}^{2n-1}_0G^2\r^{2}_1+\l\bar{\nabla}^{2n-1}_0G^3\r^{2}_0\\
&+\l {D}^{2n-1}G^4\r^{2}_{1/2}+\l{D}^{2n-1}G^5\r^{2}_{1/2}+\l\bar{D}^{2n-2}_{0}\partial_tG^5\r^{2}_{1/2}.
\end{aligned}
\end{equation*}
Firstly, by the definitions of $\bar{\mathcal{D}}^{0}_{n},~ \bar{\mathcal{D}}_{n}$ and Korn's inequality, we deduce
\begin{equation*}
\|\bar D^{2n-1}_0 u\|^2_0+\|D\bar D^{2n-1}u\|^2_1\lesssim \bar{ \mathcal{D}}_n,
\end{equation*}
and
\begin{equation}\label{5.38}
\sum^n_{j=0}\l\partial^j_t u\r^2_1\lesssim \bar {\mathcal{D}}^0_n.
\end{equation}
Summing up the the above inequalities, one has
 \begin{equation}\label{5.39}
 \|\bar D^{2n}_0u\|^2_1\lesssim \bar D_n+\bar D^0_n.
 \end{equation}
 Now, we show the estimates of $p$ and $u$. Since we have not an estimate of $\eta$ in terms of dissipation, we can not use the boundary condition on $\Sigma$ as in \eqref{5.30}. Fortunately, we can obtain higher regularity estimates of $u$ on $\Sigma$, then we have the form
\begin{equation}\label{5.40}
\left\{
\begin{aligned}
&-\Delta u+ \nabla p=-\partial_{t} u+ G^{1},~~~~~~~~~~~~~~~~~~~~~~~~~~{\rm in}~\Omega~ \\
&{\rm div} u=G^{2},~~~~~~~~~~~~~~~~~~~~~~~~~~~~~~~~~~~~~~~~~~~~~~{\rm in}~\Omega~ \\
&u=u, ~~~~~~~~~~~~~~~~~~~~~~~~~~~~~~~~~~~~~~~~~~~~~~~~~~~~ {\rm on}~\Sigma\cup\Sigma_{-1}.
\end{aligned}
\right.
\end{equation}
We apply $\partial^j_t~(j=0,1,\cdots,n-1)$ to \eqref{5.40} and employ Lemma \ref{lemmA.2} to deduce
\begin{equation}\label{5.41}
\begin{aligned}
\l\partial_{t}^{j} u\r&^{2}_{2n-2j+1} + \l\nabla \partial_{t}^{j} p\r^{2}_{2n-2j-1}\\
\lesssim& \l\partial_{t}^{j+1} u\r^{2}_{2n-2j-1} + \l \partial_{t}^{j} G^{1}\r^{2}_{2n-2j-1} +\l\partial_{t}^{j} G^{2}\r^{2}_{2n-2j}
+ \l\partial_{t}^{j} u\r^{2}_{H^{2n-2j+1/2}(\Sigma)}\\
 \lesssim &\l\partial_{t}^{j+1} u\r^{2}_{2n-2j-1}+ \l\partial_{t}^{j} u\r^{2}_{H^{2n-2j+1/2}(\Sigma)}
 + \mathcal{Y}_{n} + \bar{\mathcal{D}}_{n}+ \bar{\mathcal{D}}^0_{n}.
\end{aligned}
\end{equation}
Since $\Sigma$ and $\Sigma_{-1}$ are flat, by the definition of Sobolev norm on $T^2$ and the trace theorem, for $j=0,1,\cdots,n-1$, we have
\begin{equation}\label{5.42}
\begin{aligned}
\l\partial_{t}^{j} u\r^{2}_{H^{2n-2j+1/2}(\Sigma)}\lesssim&~ \|\partial^j_t u\|^2_{H^{1/2}(\Sigma)}+\|D^{2n-2j}\partial^j_tu\|^2_{H^{1/2}(\Sigma)}\\
\lesssim&~\|\partial^j_tu\|^2_1+\|D^{2n-2j}\partial^j_tu\|_1\\
\lesssim& ~\bar{\mathcal{D}}_{n}+ \bar{\mathcal{D}}^0_{n},
\end{aligned}
\end{equation}
where we have used the result obtained in \eqref{5.39}. Hence, we deduce
\begin{equation}\label{5.43}
\begin{aligned}
\l\partial_{t}^{j} u\r^{2}_{2n-2j+1} + \l\nabla \partial_{t}^{j} p\r^{2}_{2n-2j-1} \lesssim \l\partial_{t}^{j+1} u\r^{2}_{2n-2j-1}
 + \mathcal{Y}_{n} + \bar{\mathcal{D}}_{n}+ \bar{\mathcal{D}}^0_{n}.
\end{aligned}
\end{equation}
For the total dissipative estimate of $b$, similar in Lemma \ref{lemm5.8}, using the standard elliptic estimates to \eqref{5.32}, we obtain
\begin{equation} \label{5.44}
\begin{aligned}
\l\partial^{j}_{t} b\r_{2n-2j+1}^{2} &\lesssim \l\partial_{t}^{j+1}b\r_{2n-2j-1}^{2}+ \l\partial^{j}_{t}G^{3}\r_{2n-2j-1}^{2}\\
&\lesssim \l \partial_{t}^{j+1} b \r^{2}_{2n-2j-1} + \mathcal{Y}_{n}.
\end{aligned}
\end{equation}
Combining the estimates in \eqref{5.42} and \eqref{5.44}, one has, for $j=0,1,\cdots,n-1$,
\begin{equation}\label{5.45}
\begin{aligned}
\l\partial_{t}^{j} u\r&^{2}_{2n-2j+1} + \l\nabla \partial_{t}^{j} p\r^{2}_{2n-2j-1}+
 \l\partial^{j}_{t} b\r_{2n-2j+1}^{2}\\
 \lesssim& \l\partial_{t}^{j+1} u\r^{2}_{2n-2j-1}+\l\partial_{t}^{j+1} b\r^{2}_{2n-2j-1}
 + \mathcal{Y}_{n} + \bar{\mathcal{D}}_{n}+ \bar{\mathcal{D}}^0_{n}.
\end{aligned}
\end{equation}
After a simple induction on \eqref{5.45},  the definitions of $ \bar{\mathcal{D}}_{n}$ and $ \bar{\mathcal{D}}^0_{n}$, one has
\begin{equation}\label{5.50}
\sum_{j=0}^{n} \l\partial_{t}^{j} u\r^{2}_{2n-2j+1} + \sum_{j=0}^{n-1}\l\partial_{t}^{j}  p\r^{2}_{2n-2j}+\sum_{j=0}^{n} \l\partial_{t}^{j} b\r^{2}_{2n-2j+1}
 \lesssim \mathcal{Y}_{n} + \bar{\mathcal{D}}_{n}+\bar{\mathcal{D}}_{n}^0,
\end{equation}
where for $j=n$, we have used the result in \eqref{5.38}.

Note that the dissipation estimates in $\bar{\mathcal{D}}_{n}$ and $\bar{\mathcal{D}}_{n}^0$ only contains $u$ and $b$, then we have to recover certain dissipation estimates of $\eta$. We may derive some estimates of $\partial_t^j\eta$ for $j=0,1,\cdots,n+1$ on $\Sigma$ by employing the boundary conditions of \eqref{3.13}:
\begin{equation}\label{5.51}
\eta=p-2\partial_3u_3-G^4,
\end{equation}
and
\begin{equation}\label{5.52}
\partial_{t} \eta= u_{3} + G^{5},
\end{equation}

For $j=0$, we use the boundary condition \eqref{5.51}. Note that we do not have any bound on p on the boundary $\Sigma$, but we have bounded $\nabla p$ in $\Omega$. Thus, we differentiate \eqref{5.51} and employ $\eqref{5.50}$ to find that
\begin{equation*}
\begin{aligned}
\|D \eta\|^{2}_{2n-3/2}&\lesssim \| D p\|^{2}_{H^{2n-3/2}(\Sigma)}
+\|D\partial_{3} u_{3} \|^{2}_{H^{2n-3/2}(\Sigma)}  +\|D  G^{4}\|^{2}_{H^{2n-3/2}(\Sigma)}\\
&\lesssim \| \nabla p\|^{2}_{2n-1}+\| u_{3} \|^{2}_{2n+1}
+\| G^{4}\|^{2}_{2n-1/2}\\
&\lesssim \mathcal{Y}_{n} + \bar{\mathcal{D}}_{n}+\bar{\mathcal{D}}_{n}^0.
\end{aligned}
\end{equation*}
Thanks to the critical zero average condition
\begin{equation*}
\int_{T^{2}} \eta=0,
\end{equation*}
allow us to use Poincar${\rm\acute{e}}$ inequality on $\Sigma$ to know
\begin{equation}\label{5.53}
\|\eta\|^{2}_{2n-1/2} \lesssim \|\eta\|^{2}_{0} +  \|D \eta\|^{2}_{2n-3/2}
\lesssim \|D \eta\|^{2}_{2n-3/2} \lesssim \mathcal{Y}_{n}+\bar{\mathcal{D}}_{n}+\bar{\mathcal{D}}_{n}^0.
\end{equation}

For $j=1$, we use $\eqref{5.52}$, the definition of $\mathcal{Y}_n$ and $\eqref{5.50}$ to see
\begin{equation}\label{5.54}
\begin{aligned}
\l\partial_{t} \eta\r^{2}_{2n-1/2}&\lesssim \| u_{3}\|^{2}_{H^{2n-1/2}(\Sigma)}
+\l G^{5}\r^{2}_{H^{2n-1/2}(\Sigma)}\\
&\lesssim \| u_{3}\|^{2}_{2n}
+\l G^{5}\r^{2}_{2n-1/2}\\
&\lesssim \mathcal{Y}_{n} + \bar{\mathcal{D}}_{n}+\bar{\mathcal{D}}^0_n.
\end{aligned}
\end{equation}

Finally, for $j=2,...,n+1$ we apply $\partial_{t}^{j-1} $ to $\eqref{5.52}$ and use trace estimate to see that
\begin{equation}\label{5.55}
\begin{aligned}
\l\partial_{t}^{j} \eta\r^{2}_{2n-2j+5/2}&\lesssim \l\partial^{j-1}_{t} u_{3}\r^{2}_{H^{2n-2j+5/2}(\Sigma)}
+\l\partial^{j-1}_{t} G^{5}\r^{2}_{H^{2n-2j+5/2}(\Sigma)}\\
&\lesssim \l\partial^{j-1}_{t} u_{3}\r^{2}_{2n-2(j-1)+1}
+\l\partial^{j-1}_{t} G^{5}\r^{2}_{2n-2(j-1)+1/2}\\
&\lesssim  \mathcal{Y}_{n} + \bar{\mathcal{D}}_{n}+\bar{\mathcal{D}}^0.
\end{aligned}
\end{equation}

Summing $\eqref{5.53}$, $\eqref{5.54}$ and $\eqref{5.55}$, we complete the estimate for $\eta$, namely,
\begin{equation}\label{5.57}
\|\eta\|^{2}_{2n-1/2} + \l\partial_{t} \eta\r^{2}_{2n-1/2}+\sum_{j=2}^{n+1}\l\partial_{t}^{j} \eta\r^{2}_{2n-2j+5/2}
\lesssim \mathcal{Y}_{n} + \bar{\mathcal{D}}_{n}+\bar{\mathcal{D}}^0_n.
\end{equation}
It follows from \eqref{5.50} and \eqref{5.57} that, for $n=2N$ or $n=N+2$, we have
\begin{equation}\label{5.58}
\mathcal{D}_{n}\lesssim \bar{\mathcal{D}}_{n}+\bar{\mathcal{D}}_{n}^{0}+\mathcal{Y}_n.
\end{equation}
Setting $n=2N$ in $\eqref{5.58}$ and using the estimates $\eqref{4.8}$-$\eqref{4.9}$ in Lemma \ref{lemm5.2} to estimate $\mathcal{Y}_{2N} \lesssim (\mathcal{E}_{2N})^{\theta}\mathcal{D}_{2N} + \mathcal{E}_{N+2}\mathcal{F}_{2N}$. On the other hand, we set $n=N+2$  and apply the estimate \eqref{4.11} in Lemma \ref{lemm5.3} to bound $\mathcal{Y}_{N+2} \lesssim (\mathcal{E}_{2N})^{\theta}\mathcal{D}_{N+2}$.

\end{proof}

\subsection{Global Energy Estimates}

We first need to control $\mathcal{F}_{2N}$. This is achieved by the following proposition.

\begin{prop}\label{prop6.1}
There exists a universal constant $0<\delta<1$ so that if $\mathcal{G}_{2N} (T)\leq \delta$, then
\begin{equation}\label{6.1}
\sup_{0\leq r \leq t} \mathcal{F}_{2N}(r) \lesssim \mathcal{F}_{2N}(0) + t \int^{t}_{0} \mathcal{D}_{2N}, ~~~for~all~0\leq t \leq T.
\end{equation}
\end{prop}
\begin{proof}
Based on the transport estimate on the kinematic boundary condition, we may show as in Lemma 7.1 of \cite{Ticee} that
\begin{equation}\label{6.2}
\begin{aligned}
&\sup_{0\leq r \leq t} \mathcal{F}_{2N}(r) \lesssim \exp (C\int_{0}^{t} \sqrt{\mathcal{E}_{N+2}(r)}dr)\\
&\times \left[\mathcal{F}_{2N}(0)+ t\int_{0}^{t} (1+\mathcal{E}_{2N}(r)) \mathcal{D}_{2N}(r)dr +
\left(\int_{0}^{t} \sqrt{\mathcal{E}_{N+2}(r)\mathcal{F}_{2N}(r)} \right)^{2}\right].
\end{aligned}
\end{equation}
According to  $\mathcal{G}_{2N}\leq \delta$, we know
\begin{equation}\label{6.3}
 \int_{0}^{t} \sqrt{\mathcal{E}_{N+2}(r)}dr \lesssim \sqrt{\delta}
\int_{0}^{t}\frac{1}{(1+r)^{2N-4}} dr \lesssim \sqrt{\delta}.
\end{equation}
Since $\delta\leq 1$, this implies that for any constant $C>0$,
\begin{equation}\label{6.4}
\exp \left(C\int_{0}^{t} \sqrt{\mathcal{E}_{N+2}(r)}dr\right)\lesssim 1.
\end{equation}
Then by $\eqref{6.3}$ and $\eqref{6.4}$,  we deduce from $\eqref{6.2}$ that
\begin{equation}\label{6.5}
\begin{aligned}
\sup_{0\leq r \leq t} \mathcal{F}_{2N}(r) \lesssim& \mathcal{F}_{2N}(0)+ t\int_{0}^{t} \mathcal{D}_{2N}(r) dr +
\sup_{0\leq r\leq t} \mathcal{F}_{2N}(r) \left(\int_{0}^{t} \sqrt{\mathcal{E}_{N+2}(r)}dr\right)^{2}\\
\lesssim &\mathcal{F}_{2N}(0)+ t\int_{0}^{t} \mathcal{D}_{2N}(r) dr + \delta
\sup_{0\leq r\leq t} \mathcal{F}_{2N}(r).
\end{aligned}
\end{equation}
By taking $\delta$ small enough, $\eqref{6.1}$ follows.
\end{proof}

This bound on $\mathcal{F}_{2N}$ allows us to estimate the integral of $\mathcal{E}_{N+2}\mathcal{F}_{2N}$ and
$\sqrt{\mathcal{D}_{2N}\mathcal{E}_{N+2}\mathcal{F}_{2N}}$ as in Corollary 7.3 of \cite{Ticee}.
\begin{coro}\label{coro6.1}
There exists $0<\delta<1$ so that if $\mathcal{G}_{2N}(T) \leq \delta$, then
\begin{equation}\label{6.7}
\int_{0}^{t}  \mathcal{E}_{N+2}\mathcal{F}_{2N} \lesssim \delta \mathcal{F}_{2N}(0)+ \delta \int_{0}^{t} \mathcal{D}_{2N}(r) dr,
\end{equation}
and
\begin{equation}\label{6.6}
\int_{0}^{t}  \sqrt{\mathcal{D}_{2N} \mathcal{E}_{N+2}\mathcal{F}_{2N}} \lesssim \mathcal{F}_{2N}(0)+ \sqrt{\delta} \int_{0}^{t} \mathcal{D}_{2N}(r) dr,
\end{equation}
for all $0\leq t\leq T$.
\end{coro}

Now we show the boundness of the high-order terms.
\begin{prop}
There exists $0<\delta<1$ so that if $\mathcal{G}_{2N} (T)\leq \delta$, then
\begin{equation}\label{6.8}
\sup_{0\leq r \leq t} \mathcal{E}_{2N}(r) + \int_{0}^{t} \mathcal{D}_{2N}+  \sup_{0\leq r \leq t}
\frac{\mathcal{F}_{2N}(r)}{(1+r)}\lesssim \mathcal{E}_{2N}(0)+\mathcal{F}_{2N}(0), ~~~for~all~0\leq t \leq T.
\end{equation}
\end{prop}
\begin{proof}
Fix $0\leq t\leq T$. We sum up the results of Lemma \ref{lemm5.4} and Lemma \ref{lemm5.6} to know
\begin{equation}\label{6.9}
\begin{aligned}
&\bar{\mathcal{E}}^{0}_{2N}(t)+  \bar{\mathcal{E}}_{2N}(t)+ \int_{0}^{t} ( \bar{\mathcal{D}}^{0}_{2N}
+\bar{\mathcal{D}}_{2N})\\
\leq C_1 \mathcal{E}_{2N}(0)&+C_1(\mathcal{E}_{2N}(t))^{3/2}+
C_1\int_{0}^{t} ( \mathcal{E}^{\theta}_{2N} \mathcal{D}_{2N} + \sqrt{\mathcal{D}_{2N} \mathcal{E}_{N+2} \mathcal{F}_{2N}} ) . \end{aligned}
\end{equation}

Then, combining with Lemma \ref{lemm5.8} and Lemma \ref{lemm5.9}, we deduce
\begin{equation}\label{6.10}
\begin{aligned}
\mathcal{E}_{2N}(t)+ \int_{0}^{t} \mathcal{D}_{2N}& \leq C_2\left( \bar{\mathcal{E}}^0_{2N}+\bar{\mathcal{E}}_{2N}+
\int_{0}^{t}(\bar{\mathcal{D}}^{0}_{2N}
+\bar{\mathcal{D}}_{2N})\right)+C_2(\mathcal{E}_{2N}(t))^{1+\theta}\\
&+C_2\int^t_0 (\mathcal{E}^{\theta}_{2N} \mathcal{D}_{2N}+ \sqrt{\mathcal{D}_{2N} \mathcal{E}_{N+2} \mathcal{F}_{2N}}+ \mathcal{E}_{N+2} \mathcal{F}_{2N})\\
&\leq C_3( \mathcal{E}_{2N}(0)+\mathcal{E}_{2N}(t))^{1+\theta}+\mathcal{E}_{2N}(t))^{3/2})\\
&+C_3\int^t_0 (\mathcal{E}^{\theta}_{2N} \mathcal{D}_{2N}+ \sqrt{\mathcal{D}_{2N} \mathcal{E}_{N+2} \mathcal{F}_{2N}}+ \mathcal{E}_{N+2} \mathcal{F}_{2N}).
\end{aligned}
\end{equation}
Let us assume that $\delta\in(0,1)$ is as small as in Corollary \ref{coro6.1}, thus we conclude
\begin{equation}\label{6.12}
\sup_{0\leq r \leq t} \mathcal{E}_{2N}(t) + \int_{0}^{t}  \mathcal{D}_{2N} \lesssim \mathcal{E}_{2N}(0)+\mathcal{F}_{2N}(0).
\end{equation}
\end{proof}

It remains to show the decay estimates of $\mathcal{E}_{N+2}$. Before that, we show that the pressure term involving in Lemma \ref{lemm5.5} can be absorbed into $ \bar{\mathcal{E}}^0_{N+2} + \bar{\mathcal{E}}_{N+2} $.
\begin{lemm}\label{lemm6.1}
Let $F^2$ be defined in \eqref{3.9} with $\partial^\alpha=\partial^{N+2}_t$, then there exists a constant $\delta\in(0,1)$ so that if $\mathcal{G}_{2N}\leq \delta$, then
\begin{equation}\label{0.1}
\begin{aligned}
\frac{1}{2}(\bar{\mathcal{E}}^0_{N+2} + \bar{\mathcal{E}}_{N+2})\leq& \bar{\mathcal{E}}^0_{N+2} + \bar{\mathcal{E}}_{N+2}
-2\int_{\Omega}\partial^{N+1}_tp F^2 J\\
\leq&\frac{3}{2}(\bar{\mathcal{E}}^0_{N+2} + \bar{\mathcal{E}}_{N+2}).
\end{aligned}
\end{equation}
\end{lemm}
\begin{proof}
Let us assume that $\delta\in(0,1)$ is as small as in Corollary \ref{coro6.1}. According to  Theorem \eqref{theo5.5}, one has
\begin{equation*}
\begin{aligned}
\mathcal{E}_{N+2}\lesssim\bar{\mathcal{E}}^0_{N+2} + \bar{\mathcal{E}}_{N+2}+\mathcal{E}_{2N}^{\theta}\mathcal{E}_{N+2}
\leq C(\bar{\mathcal{E}}^0_{N+2} + \bar{\mathcal{E}}_{N+2})+C\delta\mathcal{E}_{N+2}.
\end{aligned}
\end{equation*}
Hence, we deduce
\begin{equation}\label{0.2}
\mathcal{E}_{N+2}\lesssim\bar{\mathcal{E}}^0_{N+2} + \bar{\mathcal{E}}_{N+2}.
\end{equation}
Combining the estimates obtained in \eqref{4.1} and \eqref{4.4}, we know that
\begin{equation*}
\begin{aligned}
\left|2\int_{\Omega}\partial^{N+1}_tp F^2 J\right|\leq& 2\l\partial^{N+1}_tp\r_0\|F^2\|_0\|J\|_{L^{\infty}}\\
\leq& C\sqrt{\mathcal{E}_{N+2}}\sqrt{\mathcal{E}_{2N}^\theta\mathcal{E}_{N+2}}\\
=&\displaystyle\mathcal{E}_{2N}^{\theta/2}\mathcal{E}_{N+2}\leq C\mathcal{E}_{2N}^{\theta/2}(\bar{\mathcal{E}}^0_{N+2} + \bar{\mathcal{E}}_{N+2})\\
\leq& C \delta^{\theta/2}(\bar{\mathcal{E}}^0_{N+2} + \bar{\mathcal{E}}_{N+2}).
\end{aligned}
\end{equation*}
If $\delta$ is small enough, \eqref{0.1} follows.
\end{proof}

\begin{prop}\label{prop6.3}
There exists $0<\delta<1$ so that if $\mathcal{G}_{2N} (T)\leq \delta$, then
\begin{equation}\label{6.13}
(1+t^{4N-8}) \mathcal{E}_{N+2}(t) \lesssim \mathcal{E}_{2N}(0)+ \mathcal{F}_{2N}(0)~~for~all~0\leq t\leq T.
\end{equation}
\end{prop}
\begin{proof}
Fix $0\leq t\leq T$.
 According to Lemma \ref{lemm5.5}, Lemma \ref{lemm5.7} and Lemma \eqref{lemm6.1}, we know
\begin{equation}\label{6.14}
\begin{aligned}
\partial_t\left(\bar{\mathcal{E}}^{0}_{N+2}(t)+  \bar{\mathcal{E}}_{N+2}(t)\right)+ \bar{\mathcal{D}}^{0}_{N+2}
+\bar{\mathcal{D}}_{N+2}
\leq \mathcal{E}^{\theta/2}_{2N} \mathcal{D}_{N+2}+\sqrt{\mathcal{E}_{2N}}\mathcal{D}_{N+2}.
\end{aligned}
\end{equation}
 Let us assume that $\delta\in(0,1)$ is as small as in Corollary \ref{coro6.1}, thus we have $\mathcal E_{2N}(t)\leq \mathcal{G}_{2N}(T)\leq \delta$. Similar in \eqref{0.2}, we can obtain
\begin{equation}\label{0.3}
\mathcal{D}_{N+2}\lesssim\bar{\mathcal{D}}^0_{N+2} + \bar{\mathcal{D}}_{N+2}.
\end{equation}
Thus, combining \eqref{0.2}, \eqref{0.3} and \eqref{6.14} we deduce
\begin{equation}\label{6.15}
\partial_t\mathcal{E}_{N+2}+\mathcal{D}_{N+2}\lesssim \mathcal{E}^{\theta/2}_{2N} \mathcal{D}_{N+2}+\sqrt{\mathcal{E}_{2N}}\mathcal{D}_{N+2}\lesssim \delta^{\theta/2}\mathcal{D}_{N+2}+\sqrt{\delta}\mathcal{D}_{N+2}.
\end{equation}
Hence, if $\delta$ is small enough, we obtain
\begin{equation}\label{6.16}
\partial_t\mathcal{E}_{N+2}+\mathcal{D}_{N+2}\leq 0.
\end{equation}
On the other hand, based on the Sobolev interpolation inequality we can prove
\begin{equation}\label{6.16}
\mathcal{E}_{N+2}\lesssim \mathcal{D}_{N+2}^{\theta}\mathcal{E}_{2N}^{1-\theta},~~where ~\theta=\frac{4N-8}{4N-7}.
\end{equation}
Now since we know that the boundness of high energy estimate Proposition 6.2, we get
\begin{equation}\label{6.17}
\sup_{0\leq r \leq t}\mathcal{E}_{2N}(r) \lesssim \mathcal{E}_{2N}(0)+ \mathcal{F}_{2N}(0):=\mathcal{M}_{0},~~where ~\theta=\frac{4N-8}{4N-7}.
\end{equation}
we obtain form \eqref{6.16} that
\begin{equation}\label{6.18}
\mathcal{E}_{N+2} \lesssim \mathcal{M}^{1-\theta} \mathcal{D}^{\theta}_{N+2}.
\end{equation}
Hence by \eqref{6.17} and \eqref{6.15}, there exists some constant $C_{1}>0$ such that
\begin{equation}\label{6.19}
\frac{d}{dt}\mathcal{E}_{N+2} + \frac{C_{1}}{\mathcal{M}^{s}_{0}} \mathcal{E}^{1+s}_{N+2} \lesssim 0, ~~~where
~s=\frac{1}{\theta}-1=\frac{1}{4N-8},
\end{equation}
Solving this differential inequality directly, we obtain
\begin{equation}\label{6.20}
\mathcal{E}_{N+2}(t) \lesssim \frac{\mathcal{M}_{0}}{(\mathcal{M}^{s}_{0}+ s C_{1} (\mathcal{E}_{N+2}(0))^{s} t)^{1/s}}  \mathcal{E}_{N+2}(0).
\end{equation}
Using that $\mathcal{E}_{N+2}(0)\lesssim \mathcal{M}_{0}$ and the fact $1/s=4n-8>1$, we obtain that
\begin{equation}\label{6.21}
\mathcal{E}_{N+2}(t) \lesssim \frac{\mathcal{M}_{0}}{1+ s C_{1}  t)^{1/s}} \lesssim \frac{\mathcal{M}_{0}}{1+t)^{1/s}}\lesssim \frac{\mathcal{M}_{0}}{1+t)^{4N-8}}.
\end{equation}
This implies \eqref{6.13}

\end{proof}

Now we combine proposition to arrive at our ultimate energy estimates for $\mathcal{G}_{2N}$.

\begin{theo}
There exists a universal $0<\delta<1$ so that if $\mathcal{G}_{2N}(T)\leq \delta$, then
\begin{equation}\label{6.22}
\mathcal{G}_{2N}(t) \lesssim  \mathcal{E}_{2N}(0)+ \mathcal{F}_{2N}(0)~~for~all~0\leq t \leq T.
\end{equation}
\end{theo}
\begin{proof}
The conclusion follows directly from the definition of $\mathcal{G}_{2N}$ and Proposition \ref{prop6.1}-Proposition \ref{prop6.3}.
\end{proof}

\section{Appendix A. Analytic Tools}
\subsection*{A.1 Harmonic Extension}
We define the appropriate Poisson integral in $\mathbb{T}\times(-\infty,0)$ by
\begin{equation}\label{7.1}
\mathcal{P} \eta(x)= \sum_{n\in (L_{1}^{-1} \mathbb{Z}) \times (L_{2}^{-1} \mathbb{Z})} e^{2\pi i n \cdot x'} e^{2\pi |n| x_{3}} \hat{\eta}(n),
\end{equation}
where we have written
\begin{equation*}
\hat{\eta} (n)= \int_{\Sigma} \eta(x')\frac{e^{-2\pi i n\cdot x'}}{L_{1} L_{2}} dx'.
\end{equation*}
It is well known that $\mathcal{P}:H^s(\Sigma)\rightarrow H^{s+1/2}(\mathbb{T}\times(-\infty,0))$ is a bounded linear operator for $s>0$. However, if restricted to the domain $\Omega$, one has the following result.
\begin{lemm}\label{lemmA.4}
It holds that for all $s\in\mathbb{R}$,
\begin{equation}\label{A.5}
\|\mathcal{P}f\|_s\lesssim |f|_{s-1/2}.
\end{equation}
\end{lemm}
\begin{proof}
See \cite{Guo}
\end{proof}
\subsection*{A.2 Elliptic Estimates}

\begin{lemm}\label{lemmA.1}
Suppose $u\in H^r(\Omega)$ solve
\begin{equation}\label{A.1}
\left\{
\begin{aligned}
&-\mu\Delta u=f\in H^{r-2}(\Omega),\\
&u|_{\Sigma\cup\Sigma_{-1}}=0.
\end{aligned}
\right.
\end{equation}
then for $r\geq2$, one has
\begin{equation}
\|u\|_{r}\lesssim \|f\|_{r-2}.
\end{equation}
\end{lemm}
\begin{proof}
See [agmon].
\end{proof}

\begin{lemm}\label{lemmA.2}
Suppose $(u,p)$ solve
\begin{equation}\label{A.2}
\left\{
\begin{aligned}
&-\mu\Delta u+\nabla p=\phi \in H^{r-2}(\Omega),\\
&{\rm div}u=\psi\in H^{r-1}(\Omega),\\
&(pI-\mathbb{D}(u))e_3=\alpha\in H^{r-3/2}(\Sigma),
&u|_{\Sigma_{-1}}=0.
\end{aligned}
\right.
\end{equation}
Then for $r\geq2$, one has
\begin{equation*}
\|u\|^2_{H^r}+\|p\|^2_{H^{r-1}}\lesssim \|\phi\|^2_{H^{r-2}}+\|\psi\|^2_{H^{r-1}}+\|\alpha\|^2_{H^{r-3/2}}.
\end{equation*}
\end{lemm}
\begin{proof}
See \cite{Guo}
\end{proof}

\begin{lemm}\label{A.3}
Suppose $r\geq 2$ and let $\phi \in H^{r-2}(\Omega)$, $\psi \in H^{r-1}(\Omega)$, $f_1\in H^{r-1/2}(\Sigma)$, $f_2\in H^{r-1/2}(\Sigma_{-1})$ be given such that
\begin{equation*}
\int_{\Omega}\psi=\int_{\Sigma}f_1\cdot\nu+\int_{\Sigma_{-1}}f_2\cdot\nu.
\end{equation*}
Then there exists unique  $ u\in H^{r}(\Omega)$, $p \in H^{r-1}(\Omega)$ solving
\begin{equation}\label{A.4}
\left\{
\begin{aligned}
&-\mu\Delta u+\nabla p=\phi~~{\rm in}~\Omega\\
&{\rm div}u=\psi,~~~~~~~~~~~~{\rm in}~\Omega\\
&u=f_1,~~~~~~~~~~~~~~~{\rm on}~\Sigma\\
&u=f_2,~~~~~~~~~~~~~~~{\rm on}~\Sigma_{-1}.
\end{aligned}
\right.
\end{equation}
Moreover,
\begin{equation*}
\|u\|^2_{H^r(\Omega)}+\|\nabla p\|^2_{H^{r-2}(\Omega)}\lesssim \|\phi\|^2_{H^{r-2}(\Omega)}+\|\psi\|^2_{H^{r-1}(\Omega)}+\|f_1\|^2_{H^{r-1/2}(\Sigma)}+\|f_2\|^2_{H^{r-1/2}(\Sigma_{-1})}.
\end{equation*}
\end{lemm}
\begin{proof}
See \cite{Tan}
\end{proof}

\phantomsection

\end{document}